\documentclass[11pt,reqno]{amsart}
\usepackage{amssymb,amscd,mathrsfs}
\usepackage{tikz}
\usepackage{array}
\usepackage{eucal}
\usepackage[margin=1in]{geometry}
\usepackage{todonotes}

\expandafter\let\csname ver@amsthm.sty\endcsname\relax
\let\theoremstyle\relax

\let\qedhere\relax

\usepackage{hyperref}
\usepackage{amsthm}
\usepackage[capitalise]{cleveref} 

\numberwithin{equation}{subsection}
\DeclareMathOperator{\coker}{coker}
\DeclareMathOperator{\ev}{ev}
\DeclareMathOperator{\Ext}{Ext}
\DeclareMathOperator{\gr}{gr}
\DeclareMathOperator{\opH}{H}
\DeclareMathOperator{\h}{H}
\DeclareMathOperator{\HC}{HC}
\DeclareMathOperator{\Hom}{Hom}
\DeclareMathOperator{\id}{id}
\DeclareMathOperator{\im}{im}

\DeclareMathOperator{\maxspec}{MaxSpec}

\newcommand{\C}{\mathbb{C}}
\newcommand{\N}{\mathbb{N}}
\newcommand{\Z}{\mathbb{Z}}

\newcommand{\cp}{\mathcal{P}}
\newcommand{\cv}{\mathcal{V}}

\newcommand{\fa}{\mathfrak{a}}
\newcommand\g{\mathfrak{g}}
\newcommand{\fh}{\mathfrak{h}}
\newcommand{\fri}{\mathfrak{i}}
\newcommand{\fm}{\mathfrak{m}}
\newcommand{\fr}{\mathfrak{r}}
\newcommand{\fs}{\mathfrak{s}}
\newcommand{\fsl}{\mathfrak{sl}}

\newcommand{\Cfp}{\C[f]^+}
\newcommand{\evm}{\ev_{\fm}}
\newcommand{\fre}{\fr^e}

\newcommand{\gfp}{\g[f]^+}
\newcommand{\ggA}{\g[A]}
\newcommand{\ggAp}{\g[A_+]}
\newcommand{\ghat}{\wh{\g}}
\newcommand{\gI}{\g \otimes I}
\newcommand{\gIs}{\gI/I^s}
\newcommand{\gk}{\g^{\oplus k}}
\newcommand{\gt}{\g[t]}
\newcommand{\gtp}{\gt^+}
\newcommand{\gtps}{\gtp_s}
\newcommand{\Hbul}{\opH^\bullet}
\newcommand{\hhat}{\wh{\fh}}
\newcommand{\Lg}{L(\g)}

\newcommand{\set}[1]{\left\{ #1 \right\}}
\newcommand{\subgrp}[1]{\langle #1 \rangle}
\newcommand{\wh}[1]{\widehat{ #1 }}

\newtheorem{theorem}{Theorem}[subsection]
\newtheorem{lemma}[theorem]{Lemma}
\newtheorem{proposition}[theorem]{Proposition}
\newtheorem{corollary}[theorem]{Corollary}

\crefname{theorem}{Theorem}{Theorems}
\crefname{lemma}{Lemma}{Lemmas}
\crefname{remark}{Remark}{Remarks}
\crefname{example}{Example}{Examples}
\crefname{proposition}{Proposition}{Propositions}
\crefname{corollary}{Corollary}{Corollaries}

\theoremstyle{definition}

\newtheorem{example}[theorem]{Example}
\newtheorem{remark}[theorem]{Remark}

\title{Extensions for Generalized Current Algebras}

\author{Brian D. Boe}
\address{Department of Mathematics \\
         University of Georgia \\
         Athens, GA 30602}
\email{brian@math.uga.edu}

\author{Christopher M. Drupieski}
\address{Department of Mathematical Sciences \\
         DePaul University  \\
         Chicago, IL 60614}
\email{cdrupies@depaul.edu}

\author{Tiago R. Macedo} 
\address{Department of Science and Technology\\
         Federal University of S\~ao Paulo\\
         S\~ao Jos\'e dos Campos, S\~ao Paulo, Brazil, 12.247-014\\
         \  and \ 
         Department of Mathematics and Statistics\\
         University of Ottawa\\
         Ottawa, ON K1N 6N5}
\thanks{Research of the third author was supported by FAPESP grant 2009/05887-8 and CNPq grant 232462/2014-3}
\email{tmacedo@unifesp.br}

\author{Daniel K. Nakano}
\address{Department of Mathematics \\
         University of Georgia \\
         Athens, GA 30602}
\thanks{Research of the fourth author was supported by NSF
grant DMS-1402271}
\email{nakano@math.uga.edu}
\date{\today}

\begin{document}

\begin{abstract}
Given a complex semisimple Lie algebra $\g$ and a commutative $\C$-algebra $A$, let $\ggA = \g \otimes A$ be the corresponding generalized current algebra. In this paper we explore questions involving the computation and finite-dimensionality of extension groups for finite-dimensional $\ggA$-modules. Formulas for computing $\Ext^{1}$ and $\Ext^{2}$ between simple $\ggA$-modules are presented. As an application of these methods and of the use of the first cyclic homology, we completely describe $\Ext^{2}_{\gt}(L_{1},L_{2})$ for $\g=\fsl_{2}$ when $L_{1}$ and $L_{2}$ are simple $\gt$-modules that are each given by the tensor product of two evaluation modules.
\end{abstract} 

\maketitle

\section{Introduction} 

\subsection{}

Given a complex Lie algebra $\fa$ and a commutative $\C$-algebra $A$, the Lie algebra $\fa[A] := \fa \otimes_{\C} A$ is known as a generalized current algebra. When $\g$ is a finite-dimensional complex simple Lie algebra and $A$ is either $\C[t]$ or $\C[t,t^{-1}]$, the Lie algebra $\ggA$ is well-studied for its deep connections to infinite-dimensional Lie theory. In particular, the current algebra $\gt := \g \otimes \C[t]$ is a parabolic subalgebra of the affine Kac-Moody Lie algebra associated to $\g$ (cf.\ \cite[Section 13.1]{Kumar:2002a}) and the loop algebra $\Lg := \g \otimes \C[t,t^{-1}]$ is its centerless derived subalgebra.

The Lie algebra cohomology of $\ggA$ is not well-understood. One of the main goals of this paper is to provide a better understanding of it and in particular a better understanding in the case $A=\C[t]$. A fundamental open question is the following:
\begin{equation} \label{q1}
\text{Given a finite-dimensional $\ggA$-module $M$, is $\opH^{n}(\ggA,M)$ finite-dimensional?}
\end{equation}
For $\g$ itself, that is, the case $A = \C$, it follows from results of Chevalley and Eilenberg \cite[\S24]{CE48} that $\h^n (\g, M) \cong \h^n (\g,\C) \otimes M^\g$. Also, $\h^\bullet (\g,\C)$ is isomorphic to an exterior algebra with finitely many generators \cite[Th\'eor\`eme 10.2]{koszul}. Thus $\h^n (\g, M)$ is finite-dimensional. In the case of truncated polynomial algebras, that is, when $A=\C[t]/\subgrp{t^s}$ for some  $s \geq 1$, Hanlon conjectured that the cohomology with trivial coefficients, $\h^\bullet (\ggA,\C)$, should be isomorphic to $\h^\bullet (\g,\C)^{\otimes s}$.  This is one of the two strong Macdonald conjectures, which were verified by Fishel, Grojnowski, and Teleman \cite{FGT08}.

For $A = \C[t]$, Feigin conjectured that $\h^\bullet(\gt,\C) \cong \h^\bullet (\g,\C)$; we verify this in \cref{theorem:gtpginvariants}. Another fundamental result along these lines was proved by Garland and Lepowsky. They described $\h^n (\gtp,\C)$, where $\gtp = \g \otimes t \C[t]$, as a finite direct sum of finite-dimensional irreducible $\g$-modules \cite[Theorem 8.6]{Garland:1976}. For a general commutative unital algebra $A$, Zusmanovich provided a formula for the second homology of $\g [A]$ with trivial coefficients \cite[Theorem 0.1]{Z94}. As a consequence of his result, if the first cyclic homology (see \cref{SS:cyclic}) of $A$ is finite-dimensional, then $\h^2(\ggA,\C)$ is also finite-dimensional.

In this paper we show that the finite-dimensionality of $\h^n(\ggA,M)$ depends only on the finite-dimensionality of $\h^i(\ggAp,M)^\g$ for $i < n$, where $A_+$ is the augmentation ideal of $A$ (see \cref{thm:gtcohomologyM}). We also prove the finite-dimensionality of $\h^2(\gI,\C)$ if $I = \subgrp{(t-a)(t-b)} \unlhd \C[t]$ where $a \neq b$ (see \cref{thm:H^2(goI).fd}). We suspect that $\h^n (\ggA, M)$ is finite-dimensional in general for $A=\C[t]$. The assumption that $\g$ is a finite-dimensional simple Lie algebra is essential. For example, if $\fh$ is a commutative Lie algebra, then $\fh[A]$ is also commutative and $\opH^{n}(\fh[A],\C) \cong \Hom_{\C}(\Lambda^{n}(\fh[A]),\C)$ is infinite-dimensional for any infinite-dimensional algebra $A$ and any $n>0$.

Finite-dimensional simple $\ggA$-modules have been determined and described as tensor products of evaluation modules by Chari, Fourier, and Khandai \cite[Proposition 10]{CFK}. Another open question which is investigated in the current work is the following:
\begin{equation} \label{q2}
\text{Given two finite-dimensional simple $\ggA$-modules $L_{1}$ and $L_{2}$, describe $\Ext^{n}_{\ggA}(L_{1},L_{2})$.}
\end{equation}
General results for $n=1$ have been proved by Kodera \cite{Kodera:2010} and by Neher and Savage \cite{NS11}. Some partial results for $n\geq 1$ were given by Fialowski and Malikov \cite{Fialowski:1994a}. Using results of Garland and Lepowsky, we compute $\Ext^{n}_{\gt}(L_{1},L_{2})$ for $n=1$ (recovering Kodera's and Neher and Savage's results), and for $n=2$ in terms of $\h^2(\gI,\C)$ where $I$ is an ideal of $\C[t]$ (see \cref{thm:Ext^1.goA,thm:ext2.gt,thm:self.ext2.ga}). We also explicitly describe $\h^2(\gI,\C)$ for $\g = \fsl_2$ when $I = \subgrp{(t-a)(t-b)} \unlhd \C[t]$ and $a \neq b$ (see \cref{thm:H^2.sl2oI}).

\subsection{Organization of the paper}

In Section~\ref{section:prelim} we introduce the notation and conventions that will be used throughout the paper. General results about the structure of the cohomology of $\ggA$, and in particular of $\gt$, are presented in Section~\ref{ss:hebmeaz}. These results provide the foundation for the results to follow. In Section~\ref{section:Ext1-Ext2}, general formulas for $\Ext^{1}$ and $\Ext^{2}$ between two finite-dimensional simple $\ggA$-modules are proved (see \cref{thm:Ext^1.goA,thm:ext2.gt,thm:self.ext2.ga}). In the case $n=1$, this provides an affirmative answer to \eqref{q1}, and a complete answer to \eqref{q2} by reducing the computation of $\Ext^{1}$ to the problem of decomposing tensor products of simple finite-dimensional $\g$-modules. In the case $n=2$, \cref{thm:ext2.gt,thm:self.ext2.ga} yield reductions of \eqref{q1} and \eqref{q2} to understanding the structure of the second cohomology of ideals $\gI$ with trivial coefficients.

In Sections~\ref{section:secondfd}--\ref{ss:h2.sl2} we study the second cohomology of $\gt$, $\gI$, and of their truncated versions. In Section~\ref{section:secondfd}, through some intricate calculations, we prove that the second cohomology for the ideal $\gI$ is finite-dimensional if $I = \subgrp{(t-a)(t-b)} \unlhd \C[t]$ and $a \neq b$. An interesting facet of this computation involves using work of Zusmanovich \cite{Z94} that led us to perform a detailed computation of the first cyclic homology of an augmented subalgebra of $\C[t]$ associated to the ideal $I$ (see \cref{prop:HC1}). As a consequence, $\Ext^{2}_{\gt}(L_{1},L_{2})$ is proved to be finite-dimensional when $L_{1}$ and $L_{2}$ are tensor products of at most two evaluation modules (see \cref{cor:Ext2simplesfd}). In Section~\ref{section:secondcohoprops}, the second cohomologies of $\gtp$, $\gI$, and their truncated and associated graded versions are compared using spectral sequence methods. As an application, key information about the $\g\times \g$-composition factors of $\opH^{2}(\gI,\C)$ are obtained. The spectral sequence methods used in Section~\ref{section:secondcohoprops} and their justification are presented in Appendix \ref{appendix}. 

The paper culminates in Section~\ref{ss:h2.sl2} with the computation of the $\g\times \g$-composition factors of $\opH^{2}(\gI,\C)$ and their respective multiplicities when $\g = \fsl_2$, $I=\subgrp{(t-a)(t-b)}$, and $a\neq b$ (cf.\ \cref{thm:H^2.sl2oI}). From our results one can completely describe $\Ext^{2}_{\fsl_{2}[t]}(L_{1},L_{2})$ when $L_{1}$ and $L_{2}$ are tensor products of at most two evaluation modules. 

\subsection{Acknowledgments}

The third author completed a portion of the present work as a Ph.D.\ student under the supervision of Adriano Moura and the co-supervision of Daniel Nakano, and is grateful for their support and guidance. He would also like to thank FAPESP for its financial support, which enabled him to pay for several visits to the University of Georgia (UGA), and thank as well the Department of Mathematics at UGA for their hospitality during all of his visits. Another part of the present work was completed while the third author was a postdoctoral fellow at the University of Ottawa. He would like to thank CNPq for its financial support and the Department of Mathematics and Statistics of the University of Ottawa for their hospitality.


\section{Preliminaries} \label{section:prelim}

\subsection{Notation} \label{subsection:notation}

Let $\g$ be a finite-dimensional complex simple Lie algebra. Let $\fh \subset \g$ be a Cartan subalgebra, let $\Phi$ be the root system of $\fh$ in $\g$, and let $\Delta$ be a choice of a simple system in $\Phi$. Let $\Phi^+$ be the positive system associated to $\Delta$ and let $\theta$ be the highest positive root. Let $Q$ be the root lattice of $\fh$ and let $Q^+$ be the positive monoid generated by $\Delta$ in $Q$. Let $P$ be the weight lattice associated to $Q$ and let $P^+ \subset P$ be the subset of dominant weights. Given $\lambda \in P^+$, let $V(\lambda)$ be the finite-dimensional irreducible $\g$-module of highest weight $\lambda$. Let $\geq$ be the partial order on $P$ defined by $\lambda \geq \mu$ if and only if $\lambda-\mu \in Q^+$. Let $W$ be the Weyl group associated to $\Phi$.

Let $\gt = \g \otimes \C[t]$ denote the current algebra and let $\gtp$ denote the ideal $\g \otimes t\C[t]$ in $\gt$. Here and in the sequel, $\otimes$ means $\otimes_{\C}$. More generally, let $A$ be a finitely generated commutative $\C$-algebra and let $\ggA$ be the Lie algebra with underlying vector space $\g \otimes A$ and with Lie bracket defined by $[x \otimes a,y \otimes b] = [x,y] \otimes ab$. Let $\maxspec(A)$ be the set of maximal ideals in $A$. Given $\fm \in \maxspec(A)$, let $\evm: \ggA \rightarrow \g$ be the Lie algebra homomorphism induced by the quotient map $A \rightarrow A/\fm \cong \C$. Given a $\g$-module $V$, let $\evm^* V$ be the $\ggA$-module obtained by pulling back the $\g$-module structure map for $V$ along $\evm$. For all $\fm,\fm' \in \maxspec(A)$, one has $\evm^* V(0) \cong \ev_{\fm'}^* V(0) \cong \C$ as $\ggA$-modules.

\subsection{Irreducible modules}

Let $\cp$ be the set of finitely-supported functions from $\maxspec(A)$ to $P^+$, that is, the set of functions $\pi: \maxspec(A) \rightarrow P^+$ such that $\pi(\fm) = 0$ for all but finitely many $\fm \in \maxspec(A)$. Then there exists a bijection between $\cp$ and the set of isomorphism classes of finite-dimensional irreducible $\ggA$-modules, which associates to $\pi \in \cp$ the isomorphism class of the $\ggA$-module $\cv(\pi):=\bigotimes_{\fm \in \maxspec(A)} \evm^* V(\pi(\fm))$. Note that there are only finitely many nontrivial factors in the tensor product because $\pi(\fm) =0$ for all but finitely many $\fm \in \maxspec(A)$, and that different orderings of the factors in the tensor product yield isomorphic modules.

Recall the involution $\lambda \mapsto \lambda^*$ on $P^+$ defined by $\lambda^* = -w_0 \lambda$, where $w_0$ is the longest element in the Weyl group $W$. Given $\pi \in \cp$, define $\pi^* \in \cp$ by $\pi^*(\fm) = \pi(\fm)^*$. Then the dual module $\cv(\pi)^* = \Hom_\C(\cv(\pi),\C)$ is isomorphic as a $\ggA$-module to $\cv(\pi^*)$.


\section{Cohomology for finite-dimensional modules} \label{ss:hebmeaz}

\subsection{}

If $A$ is a commutative $\C$-algebra then $\g \cong \g \otimes \C.1$ is a subalgebra of $\ggA$. For an augmented algebra $A$, let $A_{+}$ denote its augmentation ideal and let $\ggAp= \g \otimes A_{+}$.

\begin{theorem} \label{thm:gtcohomologyM}
Let $A$ be a commutative augmented $\C$-algebra and let $M$ be a $\ggA$-module that is finitely semisimple for $\g$. Then
\[
\h^n(\ggA,M) \cong \bigoplus_{i+j=n} \h^i(\ggA,\g;M) \otimes \h^j(\g,\C) \cong \bigoplus_{i+j=n} \h^i(\ggAp,M)^\g \otimes \h^j(\g,\C).
\]
\end{theorem}

\begin{proof}
By \cite[Theorem E.13]{Kumar:2002a}, there exists a spectral sequence
\begin{equation} \label{eq:relativespecseq}
E_2^{i,j} = \opH^i(\ggA,\g; \C) \otimes \opH^j(\g,\C) \Rightarrow \opH^{i+j}(\ggA,\C).
\end{equation}
Here $\Hbul(\ggA,\g;\C)$ denotes the relative Lie algebra cohomology of $\ggA$ relative to the subalgebra $\g$. The edge map $\opH^j(\ggA,\C) \rightarrow E_2^{0,j} = \opH^j(\g,\C)$ of the spectral sequence is the restriction map induced by the inclusion $\g \hookrightarrow \ggA$. Since this inclusion splits via the augmentation map $\varepsilon: \ggA \rightarrow \g$, the restriction map in cohomology $\opH^j(\ggA,\C) \rightarrow \opH^j(\g,\C)$ is a split surjection. Then it follows that the differential $d_r^{0,\bullet} : E_r^{0,\bullet} \rightarrow E_r^{r,\bullet+1-r}$ is trivial for $r \geq 2$ and that the space $E_2^{0,\bullet}$ of \eqref{eq:relativespecseq} consists of permanent cycles, i.e., $E_{2}^{0,\bullet} = E_{\infty}^{0,\bullet}$. Now using the fact that $E_{2}^{i,j} \cong E_{2}^{i,0} \otimes E_{2}^{0,j}$ and $d_2^{i,j} (a \otimes b) = d_2^{i,0}(a)\otimes b + (-1)^i a \otimes d_2^{0,j}(b)$ in \eqref{eq:relativespecseq}, and that $d_2^{\bullet,0}=0$, it follows that the spectral sequence collapses at the $E_2$-page, yielding the isomorphism
\[ 
\opH^n(\ggA,\C) \cong \bigoplus_{i+j = n} \opH^i(\ggA,\g;\C) \otimes \opH^j(\g,\C).
\]

Since $M$ is a $\ggA$-module that is finitely semisimple for $\g$, there exists by \cite[Theorem E.13]{Kumar:2002a} a spectral sequence
\begin{equation} \label{eq:relativespecseqM}
E_2^{i,j} = \opH^i(\ggA,\g;M) \otimes \opH^j(\g,\C) \Rightarrow \opH^{i+j}(\ggA,M).
\end{equation}
Moreover, \eqref{eq:relativespecseqM} is a module over \eqref{eq:relativespecseq}, and $E_2^{i,j} \cong E_2^{i,0} \otimes E_{2,\C}^{0,j}$, where $E_{2,\C}^{0,\bullet}$ denotes the space $E_2^{0,\bullet}$ in \eqref{eq:relativespecseq}, which consists of permanent cycles. Using the derivation property of the differential on \eqref{eq:relativespecseqM}, namely, $d_2^{i,j} (m \otimes r) = d_2^{i,0}(m) \otimes r + (-1)^i m \otimes d_{2,\C}^{0,j}(r)$, it follows that the spectral sequence \eqref{eq:relativespecseqM} also collapses at the $E_2$-page, and hence that
\[ 
\opH^n(\ggA,M) \cong \bigoplus_{i+j=n} \opH^i(\ggA,\g;M) \otimes \opH^j(\g,\C).
\]
Finally, it follows from applying the relative Hochschild-Serre spectral sequence (cf.\ \cite[\S6]{Evens:2013}) to the pairs $(\ggA,\g)$ and $(\ggAp,0)$ that $\opH^i(\ggA,\g;M) \cong \opH^i(\ggAp,M)^\g$.
\end{proof}

Now let $A=\C[t]$ and let $\lambda,\mu \in P^+$. Taking $M = \ev_0^* V(\lambda^{*}) \otimes \ev_0^* V(\mu)$ in \cref{thm:gtcohomologyM}, and using the fact that $\gtp$ acts trivially on $\ev_0^*V(\lambda^{*})$ and $\ev_0^*V(\mu)$, we get
\[ 
\Ext_{\gt}^n(\ev_0^*V(\lambda),\ev_0^*V(\mu)) \cong \bigoplus_{i+j=n} \Hom_\g(V(\lambda), \opH^i(\gtp,\C) \otimes V(\mu)) \otimes \opH^j(\g,\C).
\]
Using the explicit description of $\opH^\bullet(\gtp,\C)$, described below in \cref{theorem:Lepowsky}, this provides an explicit description for $\Ext_{\gt}^n(\ev_0^*V(\lambda),\ev_0^*V(\mu))$. This description was essentially known already to Fialowski and Malikov \cite[Proposition 2]{Fialowski:1994a}.

\subsection{Irreducible summands}  \label{ss:hn.goI}

Now we consider the case when $A=\C[t]$ and $I$ is an ideal of $A$. The adjoint action of $\gt$ on $\gI$ naturally induces $\gt$-module structures on $\h_n (\gI,\C)$ and $\h^n (\gI,\C)$. By restriction, we can consider these spaces as $\g$-modules. Recall that a $\g$-module is said to be finitely semisimple if it decomposes as a (possibly-infinite) direct sum of finite-dimensional irreducible modules.

\begin{lemma} \label{fin.comp.fac}
Let $f \in \C[t]$ and let $I = \subgrp{f} \unlhd \C[t]$. Then $\h_n(\gI,\C)$ is a finitely semisimple $\g$-module and each irreducible $\g$-summand of $\h_n(\gI,\C)$ is of the form $V(\lambda)$ for some $\lambda \in P^+$ with $\lambda \leq n \theta$.
\end{lemma}

\begin{proof}
First, let $\fa$ be an arbitrary complex Lie algebra. Write $U(\fa)$ for the universal enveloping algebra of $\fa$. Recall that the tensor product $U(\fa) \otimes \Lambda^\bullet(\fa)$ of $U(\fa)$ and the exterior algebra $\Lambda^\bullet(\fa)$ can be equipped with a differential in such a way that $U(\fa) \otimes \Lambda^\bullet(\fa)$ becomes a $U(\fa)$-free resolution of the trivial $\fa$-module $\C$, called the Koszul (or Chevalley--Eilenberg) resolution, with the left action of $U(\fa)$ on $U(\fa) \otimes \Lambda^\bullet(\fa)$ induced by the left multiplication of $U(\fa)$ on itself. Then $\opH_n(\fa,\C)$ is the $n$-th homology group of the chain complex $\C \otimes_{U(\fa)} (U(\fa) \otimes \Lambda^\bullet(\fa))$. Now suppose that a complex Lie algebra $\fs$ acts on $\fa$ by derivations. Then $U(\fa)$ and $\Lambda^\bullet(\fa)$ naturally become left $\fs$-modules, the Koszul differential makes $U(\fa) \otimes \Lambda^\bullet(\fa)$ into a complex of left $\fs$-modules, and $\opH_n(\fa,\C)$ inherits an $\fs$-module structure.

Specializing to the case $\fa = \gI$ and $\fs = \g$, the previous discussion implies that $\h_n(\gI,\C)$ is a $\g$-module subquotient of $\Lambda^n(\gI)$. Next, $\Lambda^n(\gI)$ is a $\g$-module quotient of $(\gI)^{\otimes n}$. Since $I = \bigoplus_{i \geq 0} \C \cdot t^i f$ as a vector space, we get a corresponding $\g$-module decomposition $\gI = \bigoplus_{i \geq 0} \g \otimes \C \cdot t^i f$. Given $i\geq 0$, set $\g_i = \g \otimes \C \cdot t^i f$. Then $\g_i$ is isomorphic as a $\g$-module to the adjoint representation of $\g$. Now $(\gI)^{\otimes n} \cong \bigoplus_{0 \leq i_1, \ldots, i_n} (\g_{i_1} \otimes \dots \otimes \g_{i_n})$ and $(\g_{i_1} \otimes \dots \otimes \g_{i_n}) \cong \g^{\otimes n} \cong V(\theta)^{\otimes n}$ as $\g$-modules. Since every finite-dimensional representation of $\g$ is completely reducible, $V(\theta)^{\otimes n}$ can be written in the form $\bigoplus_{\lambda \in P^+} V(\lambda)^{\varepsilon(\lambda)}$ with $\varepsilon(\lambda) > 0$ only if $\lambda \leq n\theta$. Combining these observations, it follows that $(\gI)^{\otimes n}$ and hence also $\Lambda^n(\gI)$ and $\h_n(\gI,\C)$ are finitely semisimple $\g$-modules, and each irreducible $\g$-summand of $\h_n (\gI,\C)$ is of the form $V(\lambda)$ for some $\lambda \in P^+$ with $\lambda \leq n \theta$.
\end{proof}

\begin{lemma} \label{lem:duality}
Let $\fa$ be a Lie algebra, let $\fri \subset \fa$ be an ideal of $\fa$, and let $M$ be an $\fa$-module. For each $n \geq 0$, $\h^n (\fri, M^*)$ is isomorphic to $\h_n (\fri, M)^*$ as an $\fa / \fri$-module.
\end{lemma}

\begin{proof}
As in the first paragraph of the proof of \cref{fin.comp.fac}, let $U(\fri) \otimes \Lambda^\bullet(\fri)$ denote the Koszul resolution for $\fri$. Then considering $M$ as a chain complex concentrated in degree $0$, the tensor product of complexes $[U(\fri) \otimes \Lambda^\bullet(\fri)] \otimes M$ is naturally a $U(\fri)$-projective resolution of $M$. Then $\opH_\bullet(\fri,M)$ is the homology of the chain complex $\C \otimes_{U(\fri)} \left( [U(\fri) \otimes \Lambda^\bullet(\fri)] \otimes M \right)$. For legibility, we identify this chain complex as a vector space with $\Lambda^\bullet(\fri) \otimes M$. The Lie algebra $\fa$ acts via the adjoint action on $\Lambda^\bullet(\fri)$ and acts as given on $M$. Then the diagonal action of $\fa$ on $\Lambda^\bullet(\fri) \otimes M$ induces the natural action of $\fa/\fri$ on $\opH_n(\fri,M)$.

Again using the Koszul resolution, $\Hbul(\fri,M^*)$ can be computed as the cohomology of the cochain complex $\Hom_{U(\fri)}(U(\fri) \otimes \Lambda^\bullet(\fri),M^*)$; when $M^*$ is understood, we will refer to this cochain complex simply as the dual Koszul complex. There are natural vector space isomorphisms
\[
\Hom_{U(\fri)}(U(\fri) \otimes \Lambda^\bullet(\fri),M^*) \cong \Hom_\C(\Lambda^\bullet(\fri),M^*) \cong \Hom_\C(\Lambda^\bullet(\fri) \otimes M,\C),
\]
and one can check that these isomorphisms are compatible with the action of $\fa$. Since $\Hom_\C(-,\C)$ is an exact (contravariant) functor, the cohomology of the cochain complex $\Hom_\C(\Lambda^\bullet(\fri) \otimes M,\C)$ is precisely $\opH_\bullet(\fri,M)^*$. Thus, for each $n \geq 0$, $\opH^n(\fri,M^*) \cong \opH_n(\fri,M)^*$ as $\fa/\fri$-modules.
\end{proof}

The next result follows directly from the previous two lemmas and from the observation that if $\lambda \in P^+$ with $\lambda \leq n\theta$ for some $n \geq 0$, then the highest weight $-w_0\lambda$ of $V(\lambda)^*$ also satisfies the inequality $-w_0\lambda \leq n\theta$.

\begin{lemma}
Let $f \in \C[t]$ and let $I = \subgrp{f} \unlhd \C[t]$. If $\h_n (\gI,\C)$ is finite-dimensional, then $\h^n(\gI,\C)$ is a finite-dimensional $\g$-module and each $\g$-composition factor of $\h^n(\gI,\C)$ is of the form $V(\lambda)$ for some $\lambda \in P^+$ with $\lambda \leq n \theta$. \qed
\end{lemma}

In Section \ref{section:secondfd} we will show that $\h_2(\gI,\C)$ is finite-dimensional if $I = \subgrp{(t-a)(t-b)}$ and $a \neq b$.

\subsection{Cohomology of \texorpdfstring{$\gtp$}{gtp}} \label{ss:Hn.gtp}

Let $\Lg = \g \otimes \C[t,t^{-1}]$ be the loop algebra. The standard realization for the affine Kac--Moody algebra $\ghat$ begins with a universal central extension $\Lg \oplus \C c$ of the $\Lg$. Then one defines a derivation $d$ of $\Lg \oplus \C c$ that maps $c$ to zero and that acts on the subspace $\g \otimes \C t^n$ of $\Lg$ as multiplication by $n$. Then $\ghat$ is the semi-direct product Lie algebra $\Lg \oplus \C c \oplus \C d$. Explicit formulas for the Lie bracket in $\ghat$ can be found in \cite[13.1.(2)]{Kumar:2002a}. Garland and Lepowsky \cite{Garland:1976} define the affine Lie algebra as $\Lg \oplus \C c$, so what we denote by $\ghat$ they denote by $\g^e$. Similarly, their $\g_S$ is our $\g$, their $\fr$ is our $\g \oplus \C c$, and their $\fre$ is our $\g \oplus \C c \oplus \C d$ (cf.\ \cite[\S\S 2--3]{Garland:1976}). We will use their notation $\fre$ for the algebra $\g \oplus \C c \oplus \C d$.

Set $\hhat = \fh \oplus \C c \oplus \C d$. This is the standard Cartan subalgebra of $\ghat$. Garland and Lepowsky denote it by $\fh^e$, while their $\fh$ is our $\fh \oplus \C c$. Let $h_1,\ldots,h_n$ be the standard basis for $\fh$ corresponding to the set of simple coroots $\alpha_1^\vee,\ldots,\alpha_n^\vee$. Let $h_\theta$ be the element of $\fh$ corresponding to the coroot $\theta^\vee$, where $\theta$ is the highest positive root in $\Phi^+$. Set $h_0 = -h_\theta + c$. The standard simple roots $\alpha_1,\ldots,\alpha_n \in \Delta$ are extended to functionals on $\hhat$ that map $c$ and $d$ to zero. Then root $\theta$ is similarly extended to a linear functional on $\hhat$. Define $\delta \in \hhat^*$ by $\delta(\fh) = \delta(c) = 0$, and $\delta(d) = 1$. Set $\alpha_0 = -\theta + \delta$. The affine root lattice is defined by $\widehat Q = \sum_{i=0}^n \Z\alpha_i$.

The fundamental dominant weights $\omega_0,\ldots,\omega_n \in \hhat^*$ are defined by $\omega_i(h_j) = \delta_{ij}$ (Kronecker delta) and $\omega_i(d) = 0$. Identifying $h_j$ with $\alpha_j^\vee$, we have $\omega_i(\alpha_j^\vee) = \delta_{ij}$. Then the set $\set{\omega_0,\omega_1,\ldots,\omega_n,\delta}$ forms a basis for $\hhat^*$. The affine weight lattice is the abelian group $X = \C \delta + \sum_{i=0}^n \Z \omega_i$, and the subset of dominant weights is $X^+ = \C \delta + \sum_{i=0}^n \N \omega_i$. Set 
\[ \textstyle
P_S = \{ \lambda \in \hhat^*: \lambda(h_i) \in \N \text{ for all } 1 \leq i \leq n \} = \C \delta + \C \omega_0 + \sum_{i=1}^n \N \omega_i.
\]

\begin{proposition} \textup{\cite[Proposition 3.1]{Garland:1976}}
There is a natural bijection, denoted $\lambda \mapsto M(\lambda)$, between $P_S$ and the set of (isomorphism classes of) finite-dimensional irreducible $\fre$-modules that are irreducible as $\g$-modules. The correspondence is described as follows: The highest weight space (relative to $\fh$) of the $\g$-module $M(\lambda)$ is $\hhat$-stable, and $\lambda$ is the resulting weight for the action of $\hhat$.
\end{proposition}

Let $V$ be a finite-dimensional irreducible $\g$-module. Then $V$ is made an irreducible $\fre$-module by having $c$ and $d$ act as zero. Conversely, every finite-dimensional irreducible $\fre$-module that restricts to $V$ as a $\g$-module can be obtained as the tensor product of $V$ and certain one-dimensional representations for $\C c$ and $\C d$.

Let $R$ be the subspace of $\hhat^*$ spanned by $\alpha_0, \ldots, \alpha_n$. For $0 \leq i \leq n$, define the linear transformation $s_i: R \rightarrow R$ by $s_i (\phi) = \phi - \phi(h_i)\alpha_i$. Then the affine Weyl group $W_a$ is the group of linear automorphisms of $R$ generated by $s_0, \ldots, s_n$. The ordinary finite Weyl group $W$ associated to $\g$ identifies with the subgroup of $W_a$ generated by $s_1, \ldots, s_n$. Let $W_a^1$ be the set of minimal length left coset representatives of $W$ in $W_a$. Let $\wh{\rho} = \omega_0 + \omega_1 + \cdots + \omega_n$, and for $w \in W_a$, set $w \cdot 0 = w \wh{\rho} - \wh{\rho}$. If $w_1,w_2 \in W_a$ and $w_1 \cdot 0 = w_2 \cdot 0$, then $w_1 = w_2$ \cite[Corollary 2.6]{Garland:1976}, and if $w \in W_a^1$, then $w \cdot 0 \in P_S$ \cite[Theorem 8.5]{Garland:1976}. In particular, the modules $M(w \cdot 0)$ for $w \in W_a^1$ are mutually non-isomorphic as $\fre$-modules.

Garland and Lepowsky computed the cohomology ring $\Hbul(\gtp,\C)$ by first studying the standard Koszul complex for $\gt^- := \g \otimes t^{-1}\C[t^{-1}]$ in order to compute the cohomology ring $\Hbul(\gt^-,\C)$. They then applied the natural involution that maps $\gt^-$ isomorphically to $\gtp$. The action of $\fre$ on $\Hbul(\gtp,\C)$ is then induced by the action of $\fre$ on the dual Koszul complex for $\gtp$.

\begin{theorem} \textup{\cite[Theorem 5.7]{Lepowsky:1979}} \label{theorem:Lepowsky}
For each $j \geq 0$, there exists an isomorphism of $\fre$-modules
\[
\opH^j(\gtp,\C) \cong \bigoplus_{\substack{w \in W_a^1, \\ \ell(w)=j}} M(w \cdot 0).
\]
\end{theorem}

Since $c$ is central in $\ghat$, it acts trivially on the Koszul complex and hence also acts trivially on $\Hbul(\gtp,\C)$. Then by Theorem \ref{theorem:Lepowsky}, $c$ acts trivially on $M(w \cdot 0)$ for each $w \in W_a^1$. For each $w \in W_a^1$, we can write $w\cdot 0 = \lambda_w - d_w \delta$ for some $\lambda_w \in P^+$ and some $d_w \geq 0$. Using \cref{theorem:Lepowsky} and the Cartan involution on $\ghat$ \cite[pp.185,190]{Lepowsky:1979}, it follows for each $j \geq 0$ that there exists a $\g$-module isomorphism
\begin{equation} \label{our.hn.gtp}
\opH^j(\gtp,\C) \cong \bigoplus_{\substack{w \in W_a^1 \\ \ell(w)=j}} V (\lambda_w)^*.
\end{equation}
The natural polynomial grading on $\gtp$, induced by the polynomial grading on $\C[t]$, induces an additional non-negative grading on $\Hbul(\gtp,\C)$ that we refer to as the $t$-degree. Then $V (\lambda_w)^*$ is concentrated in $t$-degree $d_w$.

\begin{example} \label{example:H0.gtp}
To compute $\opH^0(\gtp,\C)$, one must consider elements $w \in W^1_a$ such that $\ell (w) = 0$. Then $w=1$ and $\h^0 (\gtp,\C) \cong M(1\cdot 0) \cong \C$, concentrated in $t$-degree $0$.
\end{example}

\begin{example} \label{example:H1.gtp}
To compute $\opH^1(\gtp,\C)$, one must consider elements $w \in W^1_a$ such that $\ell (w) = 1$. Then $w = s_i$ for some $0 \leq i \leq 1$. If $1 \leq i \leq 1$, then $s_i$ is an element of the left $W$-coset of $1$ in $W_a$. Then the only element in $W^1_a$ of length $1$ is $s_0$, and $\h^1(\gtp,\C)$ is isomorphic as an $\fre$-module to $M(s_0 \cdot 0)$. Since $s_0 \cdot 0 = -\alpha_0 = \theta - \delta$, it follows that $\h^1 (\gtp,\C)$ is isomorphic as a $\g$-module to the coadjoint representation $\g^*$, concentrated in $t$-degree $1$. This can also be seen from \eqref{eq:firstcohomology}
\end{example}

\begin{example} \label{example:H2.gtp}
To compute $\opH^2(\gtp,\C)$, one must consider elements $w \in W^1_a$ such that $\ell (w) = 2$. Then $w$ must have a reduced expression of the form $s_j s_i$ for some $0 \leq i,j \leq n$ with $i \neq j$. If $1 \leq i \leq n$ or if $i=0$ and $s_j s_0 = s_0 s_j$, then $w$ is in the same left $W$-coset as an element of length smaller than $2$. Then $w$ must have the form $s_0 s_j$ for some $1 \leq j \leq n$ satisfying $\alpha_j (\alpha_0^\vee) \neq 0$. Following the conventions of Bourbaki \cite[Plates I-IX]{bourbaki68}, Table \ref{table:alphai-alpha0} lists the indices $j$ satisfying $\alpha_j (\alpha_0^\vee) \neq 0$ in each Lie type. For $0 \leq j \leq n$,
\[
s_0 \cdot (s_j \cdot 0) = s_0 \cdot (-\alpha_j) = ((1-\alpha_j (\alpha_0^{\vee}))\theta - \alpha_j) - (1-\alpha_j (\alpha_0^{\vee}))\delta.
\]
Now $\alpha_j (\alpha_0^{\vee}) = -1$ if $\ghat \not \cong \fsl_2$ and $\alpha_j (\alpha_0^{\vee}) = -2$ if $\ghat \cong \fsl_2$. Thus:
\begin{enumerate}
\item If $\ghat$ is of type $\wh{A}_1$, then $\h^2(\gtp,\C) \cong V(4)$ as $\g$-modules concentrated in $t$-degree $3$.

\item If $\ghat$ is of type $\wh{A}_n$ and $n>1$, then $\h^2(\gtp,\C) \cong V(2\theta - \alpha_1)^* \oplus V(2\theta - \alpha_n)^*$ as $\g$-modules concentrated in $t$-degree $2$.

\item In any other affine Lie type, $\h^2(\gtp,\C) \cong V(2\theta - \alpha_j)^*$ as $\g$-modules concentrated in $t$-degree $2$, with $j$ as in Table \ref{table:alphai-alpha0}.
\end{enumerate}
\end{example}

\begin{table}[htb]
\begin{tabular}{ll}
Affine Lie type & $\alpha_j (\alpha_0^\vee) \neq 0$ \\[.5ex] \hline \\[-1em]
$\wh{A}_1$ & $j=1$ \\
$\wh{A}_n, \ n\geq2$ & $j=1,n$ \\
$\wh{B}_n, \ n\geq2$ & $j=2$ \\
$\wh{C}_n, \ n\geq3$ & $j=1$ \\
$\wh{D}_n, \ n\geq4$ & $j=2$ \\
$\wh{E}_6$ & $j=2$ \\
$\wh{E}_7$ & $j=1$ \\
$\wh{E}_8$ & $j=8$ \\
$\wh{F}_4$ & $j=1$ \\
$\wh{G}_2$ & $j=2$ \\ \hline
\end{tabular}
\smallskip
\caption{}  \label{table:alphai-alpha0}
\end{table}

\subsection{Restriction map} \label{subsection:restriction.map}

According to Feigin \cite{Feigin:1980}, the restriction map $\Hbul(\gt,\C) \rightarrow \Hbul(\g,\C)$ induced by the evaluation homomorphism $\ev_0: \gt \rightarrow \g$ is a ring isomorphism. He states that this result can be deduced from the calculations of Garland and Lepowsky \cite{Garland:1976}, though he provides no details or explanation. It seems likely that Feigin's strategy would have been to take $M = \C$ in \cref{thm:gtcohomologyM}. Then the isomorphism $\opH^\bullet(\gt,\C) \cong \opH^\bullet(\g,\C)$ follows from showing for all $j \geq 1$ that $\opH^j(\gtp,\C)^\g = 0$. 

Recall that, by \cref{theorem:Lepowsky}, $c$ acts trivially on $M(w \cdot 0)$ for all $w \in W_a^1$, so if $w_1,w_2 \in W_a^1$ and $M(w_1 \cdot 0) \cong M(w_2 \cdot 0)$ as $\g$-modules, then also $M(w_1 \cdot 0) \cong M(w_2 \cdot 0)$ as $\g \oplus \C c$-modules. It now follows from the proof of \cite[Lemma 6.8]{Lepowsky:1979} that the modules $M(w \cdot 0)$ for $w \in W_a^1$ are mutually non-isomorphic as $\g$-modules.

\begin{theorem} \label{theorem:gtpginvariants} {\ }
\begin{itemize} 
\item[(a)] $\opH^i(\gtp,\C)^\g = 0$ if $i \geq 1$.
\item[(b)] The restriction map induces an isomorphism $\h^\bullet(\gt,\C) \cong \h^\bullet(\g,\C)$.
\end{itemize}
\end{theorem}

\begin{proof} By the discussion of the previous paragraph, the $M(w \cdot 0)$ for $w \in W_a^1$ are mutually non-isomorphic as $\g$-modules. Thus, the trivial $\g$-module occurs as a $\g$-summand of $\opH^i(\gtp,\C)$ only if $i=0$, where it corresponds to the identity element in $ W_a^1$. This proves part (a).  Part (b) follows from part (a) and \cref{thm:gtcohomologyM}
\end{proof}


\section{\texorpdfstring{$\Ext^{1}$ and $\Ext^{2}$ between simple modules}{Ext1 and Ext2 between simple modules}} \label{section:Ext1-Ext2}

\subsection{} \label{S:spectral}

Let $\pi,\pi' \in \cp$. Observe that since $\pi$ and $\pi'$ are finitely-supported, there exist distinct maximal ideals $\fm_1,\ldots,\fm_n \in \maxspec(A)$ such that $\pi(\fm) = 0 = \pi'(\fm)$ if $\fm \notin \set{\fm_1,\ldots,\fm_n}$. Then we can write $\cv(\pi) = \bigotimes_{i=1}^n \ev_{\fm_i}^* V(\pi(\fm_i))$ and $\cv(\pi') = \bigotimes_{i=1}^n \ev_{\fm_i}^* V(\pi'(\fm_i))$. For the rest of Section \ref{section:Ext1-Ext2}, whenever $\pi,\pi' \in \cp$ are given we will assume that $\fm_1,\ldots,\fm_n$ are distinct elements of $\maxspec(A)$ satisfying these properties.

Set $I = \fm_1 \cdots \fm_n$. Then $\gI$ is an ideal in $\ggA$ that annihilates both $\cv(\pi)$ and $\cv(\pi')$. By the Chinese Remainder Theorem, there exists a ring isomorphism
\begin{equation} \label{eq:CRT}
A/I \cong A/\fm_1 \times A/\fm_2 \times \cdots \times A/\fm_n.
\end{equation}
Then $\ggA/\gI \cong \g \otimes (A/I)$ is isomorphic as a Lie algebra to $\bigoplus_{i=1}^n \g \otimes (A/\fm_i)$. Set $\g_i = \g \otimes (A/\fm_i)$. Observe that $\g_i \cong \g$ as a Lie algebra because $A/\fm_i \cong \C$. Under this identification, $\bigoplus_{i=1}^n \g_i$ acts on $\cv(\pi) = \bigotimes_{i=1}^n \ev_{\fm_i}^* V(\pi_i)$ componentwise, i.e., if $x_1,\ldots,x_n \in \g$ and $v_1 \otimes \cdots \otimes v_n \in \cv(\pi)$, then 
\[
(x_1,\ldots,x_n).(v_1 \otimes \cdots \otimes v_n) = \sum_{i=1}^n v_1 \otimes \cdots \otimes x_i.v_i \otimes \cdots \otimes v_n .
\]

Now consider the Lyndon--Hochschild--Serre (LHS) spectral sequence for the Lie algebra $\ggA$ and the ideal $\gI$:
\begin{equation} \label{eq:LHSforI}
E_2^{i,j} = \Ext_{\ggA/\gI}^i(\cv(\pi),\Ext_{\gI}^j(\C,\cv(\pi'))) \Rightarrow \Ext_{\ggA}^{i+j}(\cv(\pi),\cv(\pi')).
\end{equation}
Notice that the term $\Ext_{\gI}^j(\C,\cv(\pi'))$ can be rewritten as $\h^{j}(\gI,\C)\otimes \cv(\pi')$ since $\gI$ acts trivially on $\cv(\pi')$. From the K\"{u}nneth formula and from the first and second Whitehead Lemmas, it follows that
\begin{align}
E_2^{1,0} &\cong \opH^1(\g^{\oplus n},\cv(\pi^*) \otimes \cv(\pi')) = 0, \ \text{and} \label{eq:LHSE10vanish} \\
E_2^{2,0} &\cong \opH^2(\g^{\oplus n},\cv(\pi^*) \otimes \cv(\pi')) = 0. \label{eq:LHSE20vanish}
\end{align}

\subsection{\texorpdfstring{$\mathbf{Ext^1}$}{Ext1}} \label{subsection:Ext1}

From our analysis of the spectral sequence \eqref{eq:LHSforI} we can now provide a formula for $\Ext_{\ggA}^1(\cv(\pi),\cv(\pi'))$.

\begin{theorem} \label{thm:Ext^1.goA}
Let $\pi, \pi'\in \cp$. Set $I = \fm_1 \cdots \fm_n$, and for each $1 \leq i \leq n$, set $\pi_i=\pi(\fm_i)$, $\pi_i'=\pi'(\fm_i)$, and set $d_i = \dim_{A/\fm_i} I/(\fm_i I)$.
\begin{itemize}
\item[(a)] If $\pi_i$ and $\pi_i'$ differ for two or more indices $1 \leq i \leq n$, then $\Ext_{\ggA}^1(\cv(\pi),\cv(\pi')) = 0$.

\item[(b)] If $\pi_i$ and $\pi_i'$ differ for precisely one index $i$, then
\[
\Ext_{\ggA}^1(\cv(\pi),\cv(\pi')) \cong 
\Hom_\g(\g \otimes V(\pi_i),V(\pi_i'))^{\oplus d_i}.
\]

\item[(c)] If $\pi=\pi'$, then $\Ext_{\ggA}^1(\cv(\pi),\cv(\pi')) \cong \bigoplus_{i=1}^n \Hom_{\g}(\g \otimes V(\pi_i),V(\pi_i))^{\oplus d_i}$.
\end{itemize}
\end{theorem}

\begin{proof} The $5$-term exact sequence of low degree terms in \eqref{eq:LHSforI} yields the isomorphism
\[
\Ext_{\ggA}^1(\cv(\pi),\cv(\pi')) \cong \Hom_{\ggA/\gI}(\cv(\pi),\opH^1(\gI,\C) \otimes \cv(\pi')).
\]

For an arbitrary Lie algebra $\fa$ over $\C$, there is a natural isomorphism 
\begin{equation} \label{eq:firstcohomology}
\opH^1(\fa,\C) \cong \Hom_\C(\fa/[\fa,\fa],\C) .
\end{equation}
Since $\g$ is semisimple, we have $[\g,\g] = \g$ and hence $[\gI,\gI] = \gI^2$. Then $\opH^1(\gI,\C)$ is isomorphic as a $\ggA/\gI$-module to $\Hom_{\C}(\g \otimes (I/I^2),\C)$. Considering $I$ as a module over $A$, we get by the Chinese Remainder Theorem for Modules an isomorphism
\begin{equation} \label{eq:CRTmodules}
I/I^2 \cong I/(\fm_1I) \times I/(\fm_2I) \times \cdots \times I/(\fm_nI),
\end{equation}
which is compatible with \eqref{eq:CRT}. Recall that $d_i = \dim_{A/\fm_i} I/(\fm_iI)$. Then $\g \otimes (I/I^2) \cong \bigoplus_{i=1}^n \g_i^{\oplus d_i}$ as a $\bigoplus_{i=1}^n \g_i$-module, i.e., $\g \otimes (I/I^2)$ is a direct sum of copies of the adjoint representations for the summands in $\bigoplus_{i=1}^n \g_i$. Then
\begin{equation} \label{eq:H1gotimesI} \textstyle
\opH^1(\gI,\C) \cong (\bigoplus_{i=1}^n \g_i^{\oplus d_i})^* \cong \bigoplus_{i=1}^n (\g_i^*)^{\oplus d_i},
\end{equation}
a direct sum of copies of the coadjoint representations for the summands in $\bigoplus_{i=1}^n \g_i$. Now applying the K\"{u}nneth formula, we get
\begin{align*}
\Ext_{\ggA}^1(\cv(\pi),\cv(\pi')) &\cong \bigoplus_{i=1}^n \Hom_{\oplus_{j=1}^n \g_j}(\cv(\pi),(\g_i^*)^{\oplus d_i} \otimes \cv(\pi')) \\
&\cong \bigoplus_{i=1}^n  \left( \Hom_{\g}(V(\pi_i),(\g^*)^{\oplus d_i} \otimes V(\pi_i')) \otimes \bigotimes_{\substack{1 \leq j \leq n \\ j \neq i}} \Hom_{\g}(V(\pi_j),V(\pi_j')) \right) \\
&\cong \bigoplus_{i=1}^n  \left( \Hom_{\g}(\g \otimes V(\pi_i),V(\pi_i'))^{\oplus d_i} \otimes \bigotimes_{\substack{1 \leq j \leq n \\ j \neq i}} \Hom_{\g}(V(\pi_j),V(\pi_j')) \right).
\end{align*}
Recall that $\Hom_\g(V(\pi_j),V(\pi_j')) \cong \C$ if $\pi_j = \pi_j'$ and is zero otherwise. Then the above calculation shows that $\Ext_{\ggA}^1(\cv(\pi),\cv(\pi'))$ is zero unless $\pi_j = \pi_j'$ for all but possibly one value of $j$, proving the first statement of the theorem. The other two statements are also immediate from this calculation.
\end{proof} 

In the special cases of the current and loop algebras, $d_i = 1$ for all $1 \leq i \leq n$. Then Theorem \ref{thm:Ext^1.goA} has the following corollary:

\begin{corollary}
Take $A$ to be either $\C[t]$ or $\C [t,t^{-1}]$. Let $\pi, \pi'\in \cp$, and for each $1 \leq i \leq n$ set $\pi_i=\pi(\fm_i)$ and $\pi_i'=\pi'(\fm_i)$.

\begin{itemize}
\item[(a)] If $\pi_i$ and $\pi_i'$ differ for two or more indices $1 \leq i \leq n$, then $\Ext_{\ggA}^1(\cv(\pi),\cv(\pi')) = 0$.

\item[(b)] If $\pi_i$ and $\pi_i'$ differ for precisely one index $i$, then
\[
\Ext_{\ggA}^1(\cv(\pi),\cv(\pi')) \cong \Hom_\g(\g \otimes V(\pi_i),V(\pi'_i)) .
\]
\item[(c)] If $\pi=\pi'$, then $\Ext_{\ggA}^1(\cv(\pi),\cv(\pi)) \cong \bigoplus_{i=1}^n \Hom_{\g}(\g \otimes V(\pi_i),V(\pi_i))$.
\hfill\qed
\end{itemize}
\end{corollary}

This recovers results of Kodera \cite[Theorem 1.2]{Kodera:2010} in a more concise way. These results were also obtained by Neher and Savage in the context of equivariant map algebras \cite[Theorems 3.7 and 3.9]{NS11}.

\subsection{\texorpdfstring{$\mathbf{Ext^2}$}{Ext2}} \label{subsection:Ext2}

We continue our analysis of the spectral sequence \eqref{eq:LHSforI} to give a description of $\Ext^{2}$ between simple modules. 

\begin{theorem} \label{thm:ext2.gt} 
Let $\pi, \pi'\in \cp$. Set $I = \fm_1 \cdots \fm_n$, and for each $1 \leq i \leq n$, set $\pi_i=\pi(\fm_i)$ and $\pi_i'=\pi'(\fm_i)$. If $\pi_{i} \neq \pi_i'$ for some $1 \leq i \leq n$, then
\[
\Ext_{\ggA}^2(\cv(\pi),\cv(\pi')) \cong \Hom_{\ggA/\gI}(\cv(\pi),\opH^2(\gI,\C) \otimes \cv(\pi')).
\]
\end{theorem}

\begin{proof}
Set $V=\cv(\pi)$, set $V'=\cv(\pi')$, and for $1 \leq i \leq n$ set $V_{i}=\ev_{\fm_i}^* V(\pi_i)$ and $V'_{i}=\ev_{\fm_i}^* V(\pi'_i)$. We consider the terms in the $E_2$-page of the spectral sequence \eqref{eq:LHSforI} that contribute to $\Ext_{\ggA}^2(V,V')$. First, $E_2^{2,0} = 0$ by \eqref{eq:LHSE20vanish}. Next, $\opH^1(\gI,\C)$ is a finite-dimensional $\ggA/\gI \cong \g^{\oplus n}$-module by \eqref{eq:H1gotimesI}, so the first Whitehead Lemma implies that $E_2^{1,1} = 0$. Then we are left to consider the space $E_2^{0,2} = \Hom_{\ggA/\gI}(V,\opH^2(\gI,\C) \otimes V')$ and its contribution to $\Ext_{\ggA}^2(V,V')$.

Since $\opH^1(\gI,\C)$ is finite-dimensional, the fist and second Whitehead Lemmas imply that $E_2^{1,1} = E_2^{2,1} = 0$. Then the differentials $d_2^{1,1}: E_2^{1,1} \to E_2^{3,0}$ and $d_2^{0,2}: E_2^{0,2} \rightarrow E_2^{2,1}$ are zero.  So $E_3^{3,0} \cong E_2^{3,0}$, $E_3^{0,2} \cong E_2^{0,2}$, and we conclude that
\[
\Ext_{\ggA}^2(V,V') \cong \ker( d_3: E_2^{0,2} \to E_2^{3,0}).
\]
It follows from the K\"{u}nneth formula and from the first and second Whitehead Lemmas that
\begin{align*}
E_2^{3,0} &= \Ext_{\ggA/\gI}^3(V_1 \otimes V_2 \otimes \cdots \otimes V_n,V_1' \otimes V_2' \otimes \cdots \otimes V_n') \\
&\cong \bigoplus_{j=1}^n \Hom_\g(V_1,V_1') \otimes \cdots \otimes \Ext_\g^3(V_j,V_j') \otimes \cdots \otimes \Hom_\g(V_n,V_n') .
\end{align*}
By hypothesis, there exists an index $i$ such that $\pi_i \neq \pi_i'$. Then $V_i \not\cong V_i'$ as $\g$-modules, hence
\[
(V_i^* \otimes V_i')^\g \cong \Hom_\g(\C,V_i^* \otimes V_i') \cong \Hom_\g(V_i,V_i') = 0.
\]
Combining this observation with \cite[\S24]{CE48}, it follows that
\[
\Ext_\g^3(V_i,V_i') \cong \Ext_\g^3(\C,V_i^* \otimes V_i') = \Ext_\g^3(\C,(V_i^* \otimes V_i')^\g) = 0,
\]
and hence that $E_2^{3,0} = 0$. Then $\Ext_{\ggA}^2(V,V') \cong E_3^{0,2}$, which completes the proof.
\end{proof}

Taking $A=\C[t]$, one obtains the following result:

\begin{corollary} \label{ext2.gt}
Let $a_1, \ldots, a_n \in \C$ with $a_i \neq a_j$ if $i \neq j$, and let $\lambda_1, \ldots, \lambda_n, \mu_1, \ldots, \mu_n \in P^+$. Set $V = \otimes_{i=1}^{n} \ev_{a_i}^* V(\lambda_i)$, set $V' = \otimes_{i=1}^{n} \ev_{a_i}^* V(\mu_i)$, and let $I = \subgrp{(t-a_1) \cdots (t-a_n)} \unlhd \C[t]$. If $\lambda_i \neq \mu_i$ for some $1 \leq i \leq n$, then $\Ext_{\gt}^2(V,V') \cong \Hom_{\gt/\gI}(V,\opH^2(\gI,\C) \otimes V')$. \qed
\end{corollary}

Now let $V = \cv (\pi) = \bigotimes_{i=1}^n \ev_{\fm_i}^* V (\pi(\fm_i))$ be a finite-dimensional simple $\ggA$-module. Since $\g$ is a finite-dimensional simple Lie algebra, the $\ggA$-module $V^* \otimes V$ decomposes into a direct sum of finite-dimensional simple modules:
\begin{equation} \label{eq:tensorproductdecomposition}
V^* \otimes V
= \bigotimes_{i=1}^n \ev_{\fm_i}^* V (\pi(\fm_i)^*) \otimes \ev^*_{\fm_i} V (\pi(\fm_i))
=\bigoplus_{k = 0}^m \cv (\rho_k).
\end{equation}
This decomposition depends only on the decomposition of each $V (\pi(\fm_i)^*) \otimes V (\pi(\fm_i))$ into simple $\g$-modules, and the trivial $\ggA$-module appears only once in this decomposition. Without loss of generality we can assume that $\cv(\rho_0)$ is the trivial $\ggA$-module. For $1 \leq k \leq m$, denote by $I_k$ the ideal $\prod_{\rho_k(\fm_j) \neq 0} \fm_j$, and notice that $\gI_k$ is the annihilator of the simple module $\cv(\rho_k)$.  We can thus extend the results of \cref{thm:ext2.gt} and \cref{ext2.gt} as follows:

\begin{theorem} \label{thm:self.ext2.ga}
Let $\pi \in \cp$. Write $\cv(\pi)^* \otimes \cv (\pi) = \bigoplus_{k=0}^m \cv(\rho_k)$ as in \eqref{eq:tensorproductdecomposition}, where $\cv(\rho_0)$ is the trivial $\ggA$-module, and for each $1 \leq k \leq m$ let $\gI_k$ be the annihilator of $\cv (\rho_k)$ in $\ggA$. Then
\[
\Ext_{\ggA}^2(\cv(\pi),\cv(\pi)) 
\cong \h^2(\ggA,\C) \oplus \bigoplus_{k=1}^m \Hom_{\ggA/\gI_k} (\C,\opH^2(\gI_k,\C) \otimes \cv (\rho_k)).
\]
\end{theorem}

\begin{proof}
Since $\cv(\pi)$ is finite-dimensional, we have
\[
\Ext^2_{\ggA}(\cv (\pi), \cv(\pi))
\cong \h^2(\ggA,\cv(\pi)^* \otimes \cv(\pi))
\cong \bigoplus_{k=0}^m \h^2(\ggA, \cv (\rho_k)).
\]
By assumption, $\cv(\rho_k)$ is trivial precisely when $k=0$. Then $\h^2(\ggA, \cv (\rho_0)) = \h^2(\ggA,\C)$, and for $1 \leq k \leq m$ we can use \cref{thm:ext2.gt} to obtain
\[
\h^2(\ggA, \cv (\rho_k))
\cong \Hom_{\ggA/\gI_k}(\C,\opH^2(\gI_k,\C) \otimes \cv(\rho_k)). \qedhere
\]
\end{proof}

Specializing to $A=\C[t]$, one obtains the following corollary:

\begin{corollary} \label{cor:self.ext2.gt}
Let $a_1, \ldots, a_n \in \C$ with $a_i \neq a_j$ if $i \neq j$, and let $\lambda_1, \ldots, \lambda_n \in P^+$. Set $V = \otimes_{i=1}^{n} \ev_{a_i}^* V(\lambda_i)$ and write $V^* \otimes V = \C \oplus \bigoplus_{k=1}^m \cv(\rho_k)$ as in \eqref{eq:tensorproductdecomposition}. For each $1 \leq k \leq m$, let $\gI_k$ be the annihilator of $\cv (\rho_k)$ in $\gt$. Then
\[
\Ext_{\gt}^2 (V, V) 
\cong \bigoplus_{k=1}^m \Hom_{\gt/\gI_k} (\C, \opH^2(\gI_k,\C) \otimes \cv (\rho_k)).
\]
\qed
\end{corollary}


\section{\texorpdfstring{Finite-dimensionality of $\h_2 (\gI,\C)$ and $\opH^2 (\gI,\C)$}{Finite-dimensionality of H2(g otimes I) and H2(g otimes I)}} \label{section:secondfd}

\subsection{}

From now through Section~\ref{ss:h2.sl2} we assume that $A = \C[t]$, $f = (t-a)(t-b) \in \C[t]$, $a \neq b$, and $I = \subgrp{f} \unlhd \C[t]$. Then by \cref{fin.comp.fac}, $\h_2(\gI,\C)$ is a finitely-semisimple $\g$-module, and each irreducible $\g$-summand of $\h_2(\gI,\C)$ is of the form $V(\lambda)$ for some dominant weight $\lambda \leq 2\theta$. In particular, up to multiplicities there are only finitely many distinct irreducible summands occurring in $\h_2(\gI,\C)$. The main result of this section is the following theorem:

\begin{theorem} \label{thm:H^2(goI).fd}
Suppose $f = (t-a)(t-b) \in \C[t]$, $a \neq b$, and $I = \subgrp{f} \unlhd \C[t]$. Then $\h_2(\gI,\C)$ and $\h^2 (\gI,\C)$ are finite-dimensional.
\end{theorem}

The two finite-dimensionality claims of the theorem are equivalent by \cref{lem:duality}. Our proof of \cref{thm:H^2(goI).fd} spans the rest of Section \ref{section:secondfd} and involves showing that each of the finitely many distinct irreducible summands occurring in $\h_2(\gI,\C)$ must have finite multiplicity.

The following corollary is a direct consequence of \cref{thm:H^2(goI).fd} and \cref{ext2.gt,cor:self.ext2.gt}. Here and throughout the rest of Section \ref{section:secondfd}, we write $\ev_a$ and $\ev_b$ to denote the evaluation homomorphisms $\C[t] \rightarrow \C$ that evaluate $t$ at $a$ and $b$, respectively.

\begin{corollary} \label{cor:Ext2simplesfd}
Let $\lambda_1, \lambda_2, \mu_1, \mu_2 \in P^+$. Let $a, b \in \C$ with $a \neq b$, and set $V = \ev^*_a V(\lambda_1) \otimes \ev^*_b V(\lambda_2)$ and $V' = \ev^*_a V(\mu_1) \otimes \ev^*_b V(\mu_2)$. Then $\Ext_{\gt}^2(V,V')$ is finite-dimensional. \hfill\qed
\end{corollary}

\subsection{}

The ideal $\gI$ admits the vector space decomposition
\[
\gI = \left( \bigoplus_{i>0} \g \otimes \C.f^i \right) \oplus \left( \bigoplus_{i>0}  \g \otimes \C . tf^i \right)
\]
that is also a $\g$-module decomposition. Set $\gfp = \bigoplus_{i>0} \g \otimes \C.f^i$. Then $\gfp$ is a Lie subalgebra of $\gI$ that is isomorphic as a Lie algebra to $\gtp$, and $\bigoplus_{i>0}  \g \otimes \C . tf^i$ is isomorphic as a $\gfp$-module to the adjoint representation of $\gfp$.

Now consider the Hochschild--Serre spectral sequence \cite[\S 2]{HS53} for $\gI$ and its subalgebra $\gfp$:
\begin{equation} \label{LHS.E1}
E^1_{p,q} \cong \h_q (\gfp, \Lambda^p( (\gI)/\gfp)) \Rightarrow \h_{p+q} (\gI,\C).
\end{equation}
We will show that all of the terms on the $E^\infty$-page of \eqref{LHS.E1} that contribute to $\h_2 (\gI,\C)$ are finite-dimensional. First, $E^1_{0,2} \cong \h_2 (\g [f]^+,\C)$ is finite-dimensional because $\gfp \cong \gtp$ as Lie algebras and $\h^2 (\gtp,\C)$ is finite-dimensional by Theorem~\ref{theorem:Lepowsky}. Then $E^\infty_{0,2}$ is also finite-dimensional.

Next we will show that $E^1_{2,0} \cong \h_0 (\gtp, \Lambda^2 (\gtp)) \cong \C \otimes_{\gtp} \Lambda^2(\gtp)$ is finite-dimensional by proving that $\Lambda^2 (\gtp)$ is a finitely-generated $\gtp$-module.

\begin{lemma}
$\Lambda^2 (\gtp)$ is a finitely generated $\gtp$-module.
\end{lemma}

\begin{proof}
The polynomial grading on $\gt$ induces a corresponding $t$-grading on $\Lambda^2 (\gtp)$. Denote the $d$-th graded component of $\Lambda^2(\gtp)$ by $\Lambda^2_{(d)}$. Note that $\Lambda^2_{(d)} = 0$ unless $d \geq 2$. We will show for all $d \geq 3$ that $\Lambda^2_{(d)}$ is generated as a $\gtp$-module by $\Lambda^2_{(2)}$.

First observe that $\gtp = \bigoplus_{i>0} \g \otimes \C.t^i$ as vector spaces. Set $\g_i = \g \otimes \C.t^i$. Then there exists a natural vector space isomorphism
\[
\Lambda^2(\gtp) \cong \left( \bigoplus_{i > 0} \Lambda^2 (\g_i) \right) \oplus \left( \bigoplus_{0<i<j} \g_i \wedge \g_j \right),
\]
with
\[
\Lambda^2_{(d)} \cong \begin{cases}
\bigoplus_{0<i<d/2} \g_{i} \wedge \g_{d-i} & \text{if $d$ is odd,} \\
\Lambda^2 (\g_{d/2}) \oplus ( \bigoplus_{0<i<d/2} \g_{i} \wedge \g_{d-i} ) & \text{if $d$ is even.}
\end{cases}
\]
We now argue by induction on $d$, treating separately the cases when $d$ is odd and when $d$ is even. Throughout the argument, let $x,y,z \in \g$.

If $D \geq 1$ and $d=2D+1$ is odd, then one can use the relation
\[
(x\otimes t^{2(D-i)+1}) \cdot ((y \otimes t^i) \wedge (z \otimes t^i)) =
(y \otimes t^i) \wedge ([x,z] \otimes t^{d-i}) - (z \otimes t^i) \wedge ([x,y] \otimes t^{d-i})
\]
to see that $\g_{i} \wedge \g_{d-i}$ is generated as a $\g$-module by $\g_{i} \wedge \g_{i}$ for $0 < i \leq D$.

If $D \geq 2$ and $d=2D$ is even, then one can use the relation
\[
(x\otimes t^{2(D-i)}) \cdot ((y \otimes t^i) \wedge (z \otimes t^i)) =
(y \otimes t^i) \wedge ([x,z] \otimes t^{d-i}) - (z \otimes t^i) \wedge ([x,y] \otimes t^{d-i})
\]
to see that $\g_{i} \wedge \g_{d-i}$ is generated as a $\g$-module by $\g_{i} \wedge \g_{i}$ for $0 < i < D$, and one can use
\[
(x\otimes t) \cdot ((y \otimes t^{D-1}) \wedge (z \otimes t^D)) =
([x,y] \otimes t^D) \wedge (z \otimes t^D) - (y \otimes t^{D-1}) \wedge ([x,z] \otimes t^{D+1})
\]
to see that $\g_{D} \wedge \g_{D}$ is generated as a $\g$-module by $\g_{D-1} \wedge \g_{D}$ and $\g_{D-1} \wedge \g_{D+1}$.

Combining the previous observations, it follows that $\Lambda^2(\gtp)$ is generated as a $\gtp$-module by $\Lambda^2_{(2)}$. Since $\Lambda^2_{(2)} \cong \g_1 \wedge \g_1 \cong \g \wedge \g$ is finite-dimensional, $\Lambda^2 (\gtp)$ is then generated as a $\gtp$-module by any finite vector space basis for $\g \wedge \g$.
\end{proof}

Since $E^1_{2,0}$ is finite-dimensional, then $E^\infty_{2,0}$ must be finite-dimensional as well.

\subsection{} \label{SS:cyclic}

Finally we consider the term $E^\infty_{1,1}$. Recall that $E^\infty_{1,1}$ is a $\g$-module subquotient of $\h_2 (\gI,\C)$, which by \cref{fin.comp.fac} is finitely semisimple and has (up to multiplicitites) finitely many distinct irreducible $\g$-summands. So to show that $E^\infty_{1,1}$ is finite-dimensional, it suffices to prove that each irreducible $\g$-summand has finite multiplicity.

We start by showing that the $\fh$-weight space $(E^1_{1,1})_\lambda$ is finite-dimensional for each $\lambda \in P^+ \setminus \{0\}$. This will imply that the multiplicity of each non-trivial irreducible $\g$-summand in $E^\infty_{1,1}$ is finite. In fact, by \cref{lem:duality} it suffices to prove the following result:

\begin{lemma} \label{lem:h^1++*.fd}
If $\lambda \in P^+ \setminus \{0\}$, then $\h^1(\gtp, (\gtp)^*)_\lambda$ is finite-dimensional.
\end{lemma}

\begin{proof}
Throughout this proof we will make the following identifications without further comment: Given vector spaces $V$ and $W$, we identify $\Hom_\C (V, W^*)$ with $(V \otimes W)^*$ and also with the set of bilinear maps $V \times W \rightarrow \C$. Given $T \in \Hom_\C (V, W^*)$ and $v \in V$, we will denote $T(v) \in W^*$ by $T(v;-)$, i.e., if $w \in W$ then $T(v)(w) = T(v;w)$.

By definition, $\h^1 (\gtp, (\gtp)^*)$ is the quotient of $\ker(d^1)$ by $\im(d^0)$, where 
\[
d^0: (\gtp)^* \rightarrow \Hom_\C (\gtp, (\gtp)^*)
\]
is defined for $f \in (\gtp)^*$ and $x,y \in \gtp$ by $d^0 (f) (x) (y) = -f ([x,y])$, and
\[
d^1 : \Hom_\C (\gtp, (\gtp)^*) \rightarrow \Hom_\C (\Lambda^2 (\gtp), (\gtp)^*)
\]
is defined for $\phi \in \Hom_\C(\gtp,(\gtp)^*)$ and $x,y,z \in \gtp$ by
\[
d^1 (\phi) (x \wedge y)(z) = -\phi([x,y];z) - \phi(y; [x,z]) + \phi(x; [y,z]).
\]
Thus, if $\varphi \in \ker (d^1)$, then it must satisfy the following relation for all $a,b,c \in \gtp$: 
\begin{equation} \label{eq:kerd1}
\varphi([a,b]; c) = \varphi(a; [b,c]) - \varphi (b; [a,c]).
\end{equation}

Fix a Chevalley basis $\{ y_\alpha, h_{\gamma}, x_\alpha : \alpha \in \Phi^+, \gamma \in \Delta \}$ for $\g$ with $x_\alpha \in \g_\alpha$, $y_\alpha \in \g_{-\alpha}$, $h_{\gamma} \in \fh$, and $[x_{\gamma}, y_{\gamma}] = h_{\gamma}$ for all $\gamma \in \Delta$. More generally, set $h_{\alpha}=[x_{\alpha},y_{\alpha}]$ for $\alpha \in \Phi^{+}$.  We will use \eqref{eq:kerd1} first to show that $\varphi( a \otimes t; b \otimes t^L)$ for any $\varphi \in \ker(d^1)$, $L \geq 2$, and certain basis elements $a, b \in \g$, can be written in terms of $\varphi( a' \otimes t^k; b' \otimes t^{L'})$ for some $a', b' \in \g$, $L'<L$, and $k = 2$ or $L$.  Namely, for any $\alpha, \beta \in \Phi^+$, $\gamma \in \Delta$, and $h \in \fh$ subject to the indicated constraints, we have:
\allowdisplaybreaks{
\begin{align*}
\varphi(x_\alpha \otimes t; x_\alpha \otimes t^L)
&= - \varphi (x_\alpha \otimes t^L; x_\alpha \otimes t) \\
\beta(h) \varphi( x_\alpha \otimes t; x_\beta \otimes t^L)
&= \varphi ([x_\beta, x_\alpha] \otimes t^2; h\otimes t^{L-1})
\tag*{$\beta(h) \neq \alpha (h) = 0$,} \\
\alpha (h) \varphi (x_\alpha \otimes t; h_\gamma \otimes t^L)
&= \alpha (h) \varphi([x_\alpha, x_\gamma] \otimes t^2; y_\gamma \otimes t^{L-1}) - \varphi([x_\gamma, [x_\alpha, y_\gamma]]\otimes t^2; h \otimes t^{L-1}) \\
& \tag*{$\alpha(h) \neq \gamma(h) = 0$,} \\
\beta(h) \varphi( x_\alpha \otimes t; y_\beta \otimes t^L)
&= \varphi ([x_\alpha, y_\beta] \otimes t^2; h\otimes t^{L-1})
\tag*{$\beta(h) \neq \alpha (h) = 0$,} \\
\beta(h) \varphi( y_\alpha \otimes t; x_\beta \otimes t^L)
&= \varphi ([x_\beta, y_\alpha] \otimes t^2; h\otimes t^{L-1})
\tag*{$\beta(h) \neq \alpha (h) = 0$,} \\
\alpha (h) \varphi (y_\alpha \otimes t; h_\gamma \otimes t^L)
&= \alpha (h) \varphi([y_\gamma, y_\alpha] \otimes t^2; x_\gamma \otimes t^{L-1}) + \varphi([y_\gamma, [x_\gamma, y_\alpha]]\otimes t^2; h \otimes t^{L-1}) \\
& \tag*{$\alpha (h) \neq \gamma (h) = 0$,} \\
\varphi(y_\alpha \otimes t; y_\alpha \otimes t^L)
&= - \varphi (y_\alpha \otimes t^L; y_\alpha \otimes t) \\
\beta(h) \varphi( y_\alpha \otimes t; y_\beta \otimes t^L)
&= \varphi ([y_\alpha, y_\beta] \otimes t^2; h\otimes t^{L-1})
\tag*{$\beta(h) \neq \alpha (h) = 0$.}
\intertext{
Now, we will use \eqref{eq:kerd1} to show that $\varphi( a \otimes t^k; b \otimes t^\ell)$ for any $\varphi \in \ker(d^1)$, basis elements $a, b \in \g$, $k\ge 2$, and $\ell \ge 1$, can be written in terms of $\varphi( h \otimes t; b' \otimes t^{\ell'})$ for some $h \in \mathfrak{h}$, $b' \in \g$ and $\ell'>1$.  Namely, for any $\alpha,\beta \in \Phi^+$, $\gamma \in \Delta$, and $h \in \fh$ subject to the indicated constraints, we have:
}
\varphi(x_\alpha \otimes t^k; x_\alpha \otimes t^\ell) 
&= 0
\tag*{$k + \ell \geq 4$,} \\
\alpha(h) \varphi(x_\alpha \otimes t^k; x_\beta \otimes t^\ell) 
&= \varphi(h \otimes t; [x_\alpha, x_\beta] \otimes t^{k + \ell -1})
\tag*{$\alpha(h) \neq \beta(h) = 0$,} \\
\varphi(x_\alpha \otimes t^k; h_\gamma \otimes t^\ell) 
&= - \varphi(h_\gamma \otimes t; x_\alpha \otimes t^{k+\ell-1})
\tag*{$\alpha(h_\gamma) \neq 0$,} \\
\alpha(h) \varphi(x_\alpha \otimes t^k; y_\beta \otimes t^\ell) 
&= \varphi(h \otimes t; [x_\alpha, y_\beta] \otimes t^{k+\ell-1})
\tag*{$\alpha(h) \neq \beta(h) = 0$,} \\
\gamma(h) \varphi(h_\gamma \otimes t^k; x_\beta \otimes t^\ell)
&= \varphi(h \otimes t; [x_\gamma, [y_\gamma, x_\beta]] \otimes t^{k+\ell-1}) - \gamma (h) \varphi(y_\gamma \otimes t; [x_\gamma, x_\beta] \otimes t^{k+\ell-1}) \\
& \tag*{$\beta(h) = \gamma(h) \neq 0$, $k \geq 3$,} \\
\gamma(h) \varphi(h_\gamma \otimes t^k; y_\beta \otimes t^\ell)
&= \varphi(h \otimes t; [y_\gamma, [x_\gamma, y_\beta]] \otimes t^{k+\ell-1}) + \gamma (h) \varphi (x_\gamma \otimes t; [y_\gamma, y_\beta] \otimes t^{k+\ell-1}) \\
& \tag*{$\beta(h) = \gamma(h) \neq 0$, $k \geq 3$,} \\
\alpha(h) \varphi(y_\alpha \otimes t^k; x_\beta \otimes t^\ell) 
&= - \varphi(h \otimes t; [y_\alpha, x_\beta] \otimes t^{k + \ell -1}), 
\tag*{$\alpha(h) \neq \beta(h) = 0$,} \\
\varphi(y_\alpha \otimes t^k; h_\gamma \otimes t^\ell) 
&= - \varphi(h_\gamma \otimes t; y_\alpha \otimes t^{k+\ell-1}) 
\tag*{$\alpha (h_\gamma) \neq 0$,} \\
\varphi(y_\alpha \otimes t^k; y_\alpha \otimes t^\ell) 
&= 0
\tag*{$k + \ell \geq 4$,} \\
\alpha(h) \varphi(y_\alpha \otimes t^k; y_\beta \otimes t^\ell) 
&= - \varphi(h \otimes t; [y_\alpha, y_\beta] \otimes t^{k+\ell-1})
\tag*{$\alpha(h) \neq \beta(h) = 0$.}
\end{align*}
}

Suppose $\varphi \in \ker(d^1)_\lambda$ for some $\lambda \neq 0$. Then by weight considerations,
\[
\varphi(h \otimes t^k; h' \otimes t^\ell)
= \varphi(x_\alpha \otimes t^k; y_\alpha \otimes t^\ell)
= \varphi(y_\alpha \otimes t^k; x_\alpha \otimes t^\ell)
= 0
\]
for all $h, h' \in \fh$, $\alpha \in \Phi^+$, and $k,\ell \ge 1$. Hence, using the Chevalley basis and relations above, $\varphi$ can be completely determined by the following constants (for all $\alpha, \beta \in \Phi^{+}$, $\gamma \in \Delta$, $k \geq 1$):
\begin{gather*}
\varphi(h_{\gamma} \otimes t; x_\alpha \otimes t^k) \ \text{and} \ 
\varphi(h_{\gamma} \otimes t; y_\alpha \otimes t^k); \\
\varphi(x_\alpha \otimes t; x_\beta \otimes t), \ 
\varphi(x_\alpha \otimes t; h_{\gamma} \otimes t), \ \text{and} \ 
\varphi(x_\alpha \otimes t; y_\beta \otimes t); \\
\varphi(y_\alpha \otimes t; x_\beta \otimes t), \ 
\varphi(y_\alpha \otimes t; h_{\gamma} \otimes t), \ 
\text{and} \ 
\varphi(y_\alpha \otimes t; y_\beta \otimes t);  \\
\varphi(x_\alpha \otimes t; x_\alpha \otimes t^2)\ \text{and} \ 
\varphi(y_\alpha \otimes t; y_\alpha \otimes t^2).
\end{gather*}

Let $\xi$ be any element in $(\gtp)^*$ satisfying
\[
\xi (x_\alpha \otimes t^k) 
= \frac 1 2 \varphi(h_\alpha \otimes t; x_\alpha \otimes t^{k-1})
\quad \text{and} \quad
\xi (y_\alpha \otimes t^k) 
= -\frac 1 2 \varphi(h_\alpha \otimes t; y_\alpha \otimes t^{k-1})
\]
for all $\alpha \in \Phi^+$ and $k \ge 2$. Observe that, by the construction of $\xi$ and the relations above,
\[
(d^0 \xi) (h_{\gamma} \otimes t; x_\alpha \otimes t^k)
= \varphi(h_{\gamma} \otimes t; x_\alpha \otimes t^k) \quad \text{and} \quad
(d^0 \xi) (h_{\gamma} \otimes t; y_\alpha \otimes t^k)
= \varphi(h_{\gamma} \otimes t; y_\alpha \otimes t^k)
\]
for all $\alpha \in \Phi^+$, $\gamma \in \Delta$, and $k\geq 2$. This leaves us with finitely many possibly linearly independent elements in $\ker(d^1)/\im(d^0)$, namely the duals of the following basis elements in $\gtp \otimes \gtp$:
\begin{gather*}
(h_{\gamma} \otimes t) \otimes (x_\alpha \otimes t)\ 
\text{and} \ 
(h_{\gamma} \otimes t) \otimes (y_\alpha \otimes t);\\
(x_\alpha \otimes t) \otimes (x_\beta \otimes t),\ 
(x_\alpha \otimes t) \otimes (h_{\gamma} \otimes t),\ 
\text{and} \ 
(x_\alpha \otimes t) \otimes (y_\beta \otimes t);\\
(y_\alpha \otimes t) \otimes (x_\beta \otimes t),\ 
(y_\alpha \otimes t) \otimes (h_{\gamma} \otimes t),\ 
\text{and} \ 
(y_\alpha \otimes t) \otimes (y_\beta \otimes t);\\
(x_\alpha \otimes t) \otimes (x_\alpha \otimes t^2)\ 
\text{and} \ 
(y_\alpha \otimes t) \otimes (y_\alpha \otimes t^2);
\end{gather*}
for $\alpha, \beta \in \Phi^+$ and $\gamma \in \Delta$. Thus we conclude for each nonzero $\lambda \in P^+$ that $\h^1 (\gtp, (\gtp)^*)_\lambda$ is finite-dimensional.
\end{proof}

As a consequence of \cref{lem:h^1++*.fd} and the previous observations on the finite-dimensionality of $E_{0,2}^\infty$ and $E_{2,0}^\infty$, we can now conclude that, as a $\g$-module, $\h_2(\gI,\C)$ can be written in the form $V \oplus W$, where $V$ is a (perhaps trivial) finite direct sum of nontrivial irreducible $\g$-modules and $W$ is a (perhaps infinite-dimensional) trivial $\g$-module. Then by \cref{lem:duality}, $\h^2(\gI,\C)$ is isomorphic as a $\g$-module to $V^* \oplus W^*$. In particular, $W^*$ is a trivial $\g$-module summand of $\h^2(\gI,\C)$. To complete the proof that $\h_2(\gI,\C)$ is finite-dimensional, we will now show that $\h^2(\gI,\C)^\g$ is finite-dimensional (and hence so is $W$).

The algebra structure on $\C[t]$ induces on $A:= \C \oplus I$ the structure of an associative $\C$-algebra; we consider $A$ as an augmented algebra with augmentation ideal $I$. Since $\g$ is a simple complex Lie algebra, one has $\opH^0(\g,\C) = \C$ and $\opH^1(\g,\C) = \opH^2(\g,\C) = 0$. Then by Theorem \ref{thm:gtcohomologyM},
\[
\h^2 (\ggA,\C) \cong \h^2(\gI,\C)^\g,
\]
so we will turn our attention to $\h^2 (\ggA,\C)$. By Lemma \ref{lem:duality}, $\h^2 (\ggA,\C)$ and $\h_2(\ggA,\C)$ are either both infinite-dimensional or both of the same finite dimension. Using \cite[Theorem 0.1]{Z94} and the fact that $\g$ is simple, one has
\begin{equation} \label{eq:Zus}
\h_2 (\ggA,\C) \cong B(\g) \otimes \HC_1 (A),
\end{equation}
where $B(\g)$ is the space of coinvariants for the adjoint action of $\g$ on the second symmetric power $S^2(\g)$, and $\HC_1 (A)$ is the first cyclic homology of $A$ (defined below). Since $\g$ is a finite-dimensional simple Lie algebra, we have $S^2(\g) / [\g \cdot S^2(\g)] \cong \C$ (see, for instance, \cite[\textsection 0.3]{FGT08}), giving $\h_2 (\ggA,\C) \cong \HC_1 (A)$. 

In the following result we show that $\HC_1 (A)$ is finite-dimensional. Consequently, $\opH^2(\gI,\C)^\g$ is finite-dimensional and the multiplicity of the trivial $\g$-module in $E^\infty_{1,1}$ is finite, finishing the proof of \cref{thm:H^2(goI).fd}.

\begin{proposition} \label{prop:HC1}
Let $a_1,a_2 \in \C$ be distinct, and set $f = (t-a_1)(t-a_2) \in \C[t]$. If $I = \subgrp{f} \unlhd \C[t]$ and $A$ is the subalgebra $\C \oplus I$ of $\C[t]$, then $\dim \HC_1 (A) = 2$.
\end{proposition}

\begin{proof}
Set $\C[f] = \bigoplus_{i \geq 0} \C \cdot f^i$ and set $\Cfp = \bigoplus_{i \geq 1} \C \cdot f^i \subset I$. Then $I = \Cfp \oplus (t-a_1) \Cfp$. By definition, $\HC_1(A) = \Lambda^2(A) / T(A)$, where $T(A)$ is the subspace of $\Lambda^2(A)$ spanned by the set
\[
\{ ab \wedge c + bc \wedge a + ca \wedge b : a,b,c \in A \}.
\]
Since $A \cong \C[f] \oplus (t-a_1) \Cfp$, we can decompose $\Lambda^2(A)$ to obtain
\begin{equation} \label{eq:HC1A}
\HC_1(A) \cong \dfrac{\Lambda^2 (\C[f]) \oplus [\C[f] \wedge (t-a_1) \Cfp] \oplus \Lambda^2 ((t-a_1) \Cfp)}{T(A)}.
\end{equation}
We will study each of the terms in the numerator of \eqref{eq:HC1A} separately.

By \cref{lem:duality}, \cref{theorem:gtpginvariants}, and the second Whitehead Lemma, one has $\h_2 (\gt,\C) \cong \h_2(\g,\C) = 0$. But $\C[f] \cong \C[t]$ as algebras, so $\h_2(\g[\C[f]],\C) = 0$ as well. Applying \eqref{eq:Zus}, it then follows that $\HC_1(\C[f]) = 0$, and hence that $\Lambda^2(\C[f]) = T(\C[f]) \subseteq T(A)$. Now to finish the proof we will show that $\Lambda^2(A)$ is spanned modulo $T(A)$ by the set
\[
\{(t-a_1)^3 (t-a_2)^2 \wedge (t-a_1) (t-a_2), \ (t-a_1)^2 (t-a_2) \wedge (t-a_1) (t-a_2)\}.
\]

Without loss of generality we may assume that $a_1 = 0$ and $a_2 = 1$. For each $i,j>0$, denote the element $t^i (t-1)^j \in A$ by $(i,j)$, and denote by $H$ the subspace of $\Lambda^2 (A)$ spanned by $(3,2)\wedge(1,1)$ and $(2,1)\wedge(1,1)$. Given $(i,j)\wedge(k,\ell) \in \Lambda^2 (A)$, define its degree as $i+j+k+\ell$. We will argue by induction on the degree to show for all $i,j,k,\ell > 0$ that $(i,j)\wedge(k,\ell) \in T(A) + H$.

Note that $(i,i)$ and $(j,j)$ are elements of $\C[f]$. Then $(i,i) \wedge (j,j) \in T(A)$ by the observation that $\Lambda^2(\C[f]) \subseteq T(A)$. Now observe that 
\begin{equation} \label{eq:lower.rel}
(i+1,j) = (i,j+1) + (i,j) \quad \text{for all $i,j>0$}.
\end{equation}
Then it is enough to prove that $T(A)+H$ contains every element of the form $(i+1,i)\wedge(j+1,j)$ with $1 \leq j < i$, and every element of the form $(i+1,i)\wedge(j,j)$ with $1 \leq i, j$. In the sequel we will use the relation \eqref{eq:lower.rel} to consider only elements of either one of these two forms.

We start with degree 5. In this case, the only possible element is $(2,1)\wedge(1,1)$, which is an element of $H$ but not of $T(A)$, because any non-zero element of the form $ab\wedge c + bc \wedge a + ca\wedge b$ must have degree at least 7. In degree 6, the only possible element is $(2,2)\wedge(1,1)$, which is an element of $T(A)$ as noted above.

The only possible elements of degree 7 are $(3,2)\wedge(1,1)$ and $(2,1) \wedge (2,2)$. Since the only non-zero element of degree 7 of the form $ab\wedge c + bc \wedge a + ca\wedge b$ is
\[
(2,1)(1,1)\wedge(1,1) + (1,1)(1,1)\wedge(2,1) + (1,1)(2,1) \wedge (1,1) 
= 2 (3,2)\wedge (1,1) - (2,1)\wedge(2,2),
\]
neither $(3,2)\wedge (1,1)$ nor $(2,1)\wedge(2,2)$ is an element of $T(A)$. Instead $(3,2)\wedge(1,1) \in H$ and $(2,1)\wedge(2,2) \equiv 2 (3,2)\wedge(1,1)$ modulo $T(A)$.

Now we start the induction process. If $d=2D$ is even, then all possible elements are
\[
(D-1,D-2)\wedge(2,1), \ 
(D-2,D-3)\wedge(3,2), \ 
\ldots, \ 
(\lceil D/2 \rceil+1, \lceil D/2\rceil)\wedge(\lfloor D/2 \rfloor, \lfloor D/2 \rfloor-1).
\]
An element of the form $(i+1,i)\wedge(j+1,j)$ with $i>j\geq 1$ can be obtained from $ab\wedge c + bc \wedge a + ca\wedge b$ by taking $a=b=(j+1,j)$, $c=(i-j,i-j)$, and using \eqref{eq:lower.rel} and the induction hypothesis.

If $d = 2D+1$ is odd, then all the possible elements are 
\begin{equation} \label{possible.odd}
(D, D-1)\wedge(1,1), \
(D-1,D-2)\wedge(2,2), \
\ldots, \
(2,1)\wedge(D-1,D-1).
\end{equation}
To obtain these elements we can use linear combinations of the following $D-1$ elements in $T(A)$:
\begin{center}
\renewcommand{\arraystretch}{1.5}
\begin{tabular}{ >{\small $}c<{$} | >{\small $}c<{$} | >{\small $}c<{$} | >{\footnotesize $}c<{$}}
 a & b & c & (ab\wedge c + bc \wedge a + ca\wedge b) \\ \hline
(D-1, D-2) & (1,1) & (1,1) & 2(D,D-1)\wedge(1,1) - (D-1, D-2) \wedge(2,2) \\
(D-2, D-3) & (2,2) & (1,1) & (D,D-1)\wedge(1,1) + (D-1,D-2)\wedge (2,2) - (D-2, D-3) \wedge(3,3) \\
\vdots & \vdots & \vdots & \vdots \\
(2, 1) & (D-2,D-2) & (1,1) & (D,D-1)\wedge(1,1) + (3,2)\wedge (D-2, D-2) - (2, 1) \wedge (D-1,D-1) \\
(D-2, D-3) & (2,1) & (2,1) & 2(D,D-2)\wedge(2,1) + (4, 2) \wedge (D-2,D-3)
\end{tabular}
\end{center}
Observe that, by the induction hypothesis,
\[
2(D,D-2)\wedge(2,1) + (4, 2) \wedge (D-2,D-3) 
\equiv -2(2,1)\wedge(D-1,D-1)-(D-2,D-3)\wedge(3,3)
\]
modulo $T(A)+H$. We can organize these terms into the $(D-1) \times (D-1)$ matrix
\[
M=
\begin{pmatrix}
2 & -1 & 0\\
1 & 1 & -1  & & \cdots \\
1 & 0 & 1 & -1 \\
& \vdots & & \ddots \\
1 & 0 & 0 & \cdots & 1 & -1 \\
0 & 0 & -1 & & 0 & -2 
\end{pmatrix},
\]
where the $(i,j)$-entry corresponds to the coefficient of the term $(D-j+1, D-j)\wedge(j,j)$ on the $i$-th row of the above table. One can prove by induction on $D$ that $\det M = -d$, thus showing that every possible element listed in \eqref{possible.odd} is contained in $T(A)+H$.
\end{proof}

\section{Second Cohomology: Properties and Calculations} \label{section:secondcohoprops}

\subsection{Low-degree cohomology of \texorpdfstring{$\gtps$}{g[t]+s}} \label{ss:ldcahogtps}

Let $s>1$ and set $\gtps = \g \otimes (t \C[t]/ t^s \C[t])$. We start this section by describing $\h^i (\gtps,\C)$ for $i\leq 2$. We already know that $\h^0 (\gtps,\C) \cong \C$, and by \eqref{eq:firstcohomology} there are $\g$-module isomorphisms
\[
\h^1 (\gtps,\C) 
\cong \left( \frac{\gtps}{\left[ \gtps,\gtps \right]} \right)^* 
\cong  \left( \g \otimes \frac{t \C[t] / t^s \C[t]}{t^2 \C[t] / t^s \C[t]} \right)^*
\cong (\g \otimes \C .t)^*
\cong \g^*.
\]

\begin{proposition} \label{prop:4}
There exists an integer $s_2 > 0$ such that, for all $s > s_2$, there exists a $\g$-module isomorphism $\h^2 (\gtps,\C) \cong \h^2(\gtp,\C) \oplus \g^*$.
\end{proposition}

Recall that an explicit description of $\h^2 (\gtp,\C)$ as a $\g$-module was given in \eqref{our.hn.gtp}, so \cref{prop:4} provides an explicit computation of $\h^2 (\gtp_s,\C)$ provided $s$ is sufficiently large.

To prove \cref{prop:4}, consider the following LHS spectral sequence:
\begin{equation} \label{eq:LHS.gtp/gtps}
E^{p,q}_2 \cong \h^p (\gtps, \h^q(\g \otimes t^s\C[t],\C)) \Rightarrow \h^{p+q} (\gtp,\C).
\end{equation}
By \eqref{eq:firstcohomology}, there exist $\g$-module isomorphisms $E_2^{1,0} \cong \h^1 (\gtps,\C) \cong \g^* \cong \h^1 (\gtp,\C)$, and also
\begin{equation} \label{eq:degreescoadjoint}
E_2^{0,1}
\cong \h^0 (\gtps, \h^1 (\g \otimes t^s \C[t],\C))
\cong \h^0 (\gtps, [\g \otimes (t^s \C[t] / t^{2s} \C[t])]^*)
\cong (\g \otimes \C.t^s)^*
\cong \g^*.
\end{equation}
Thus the associated 5-term exact sequence for \eqref{eq:LHS.gtp/gtps} yields an exact sequence of $\g$-modules
\[
0 \rightarrow \g^* \rightarrow \h^2 (\gtps,\C) \rightarrow \h^2 (\gtp,\C).
\]
Since $\g$ is a finite-dimensional simple Lie algebra, all short exact sequence of finite-dimensional $\g$-modules split. Then to prove \cref{prop:4}, it suffices to show for all $s$ sufficiently large that the inflation map $\opH^2(\gtps,\C) \rightarrow \opH^2(\gtp,\C)$ is surjective. By \cref{lem:duality}, this is equivalent to showing for all $s$ sufficiently large that the quotient map $\gtp \rightarrow \gtps$ induces an exact sequence
\[
0 \rightarrow \h_2 (\gtp,\C) \rightarrow \h_2 (\gtps,\C) \rightarrow \g \rightarrow 0.
\]
This sequence is guaranteed to be right exact by considering the 5-term exact sequence of low degree terms for the Hochschild--Serre spectral sequence
\begin{equation} \label{eq:HSspecseqgtp}
E_{p,q}^2 \cong \h_p (\gtps, \h_q(\g \otimes t^s\C[t],\C)) \Rightarrow \h_{p+q} (\gtp,\C).
\end{equation}
In the next lemma we show that the edge map $\pi_{s*} : \h_2 (\gtp,\C) \rightarrow \h_2 (\gtps,\C)$ of \eqref{eq:HSspecseqgtp}, which is induced by the quotient map $\pi_s: \gtp \rightarrow \gtps$, is injective if $s$ is sufficiently large; this will complete the proof of \cref{prop:4}.

\begin{lemma} \label{lem:injhp}
For each $p> 0$, there exists $s_p > 1$ such that the edge map $\pi_{s*} : \h_p (\gtp,\C) \rightarrow \h_p (\gtps,\C)$ of \eqref{eq:HSspecseqgtp} is injective for all $s \geq s_p$.
\end{lemma}

\begin{proof}
Denote the Koszul complex associated to $\gtp$ by $\left( \Lambda^\bullet (\gtp), \partial_\bullet \right)$, and for each $s>1$ denote the Koszul complex associated to $\gtps$ by $\left( \Lambda^\bullet (\gtps), \overline{\partial}_\bullet \right)$.\footnote{As in the proof of \cref{lem:duality}, we make the identification $\C \otimes_{U(\fa)} (U(\fa) \otimes \Lambda^\bullet(\fa)) \cong \Lambda^\bullet(\fa)$.} Write $\Lambda^\bullet \pi_s : \Lambda^\bullet (\gtp) \rightarrow \Lambda^\bullet (\gtps)$ for the map of complexes induced by $\pi_s$. Next consider the polynomial grading on $\gtp$ induced by the polynomial grading on $\C[t]$. One has $\gtp = \bigoplus_{d>0} \gtp_{(d)}$ where $\gtp_{(d)} := \g \otimes \C . t^d$. The polynomial grading on $\gtp$ induces corresponding polynomial gradings on $\Lambda^\bullet(\gtp)$, $\im (\partial_{\bullet})$, $\ker (\partial_\bullet)$ and $\h_\bullet (\gtp,\C)$. The polynomial grading also induces a decreasing filtration on $\gtp$,
\[
\gtp = \g \otimes t \C[t] \supset \g \otimes t^2 \C[t] \supset \cdots \supset \g \otimes t^m \C[t] \supset \cdots,
\]
as well as a decreasing filtration on its Koszul complex:
\[ \textstyle
F_p \Lambda^n := F_p \Lambda^n(\gtp) := \sum_{d\geq p} \Lambda^n(\gtp)_{(d)}.
\]

Given $p \geq 0$, $\h_p (\gtp,\C)$ is finite-dimensional by \cite[Theorem 8.6]{Garland:1976}, so there exists $s_p>1$ such that $\h_p(\gtp,\C)_{(d)} = 0$ for all $d \geq s_p$. That is, $\im (\partial_{p+1}) \cap (\Lambda^p (\gtp))_{(d)} = \ker (\partial_p) \cap (\Lambda^p (\gtp))_{(d)}$ for all $d \ge s_p$. Thus, 
\[
F_d \Lambda^p \cap \im (\partial_{p+1}) = F_d \Lambda^p \cap \ker (\partial_{p}) \quad \text{for all $d \geq s_p$.}
\]

Fix $p \ge 0$, let $s_{p}$ be as above, and let $s \ge s_{p}$. Also suppose that $\ker(\pi_{s*}) \neq \{ 0 \}$. This means that there exists $0 \neq h_p \in \h_p (\gtp,\C)$ such that $\pi_{s*} (h_p) = 0$; that is, $h_p = k_p + \im (\partial_{p+1})$ for some $k_p \in \ker (\partial_p) \setminus \im (\partial_{p+1})$ such that $\Lambda^p \pi_s (k_p) \in \im (\overline{\partial}_{p+1})$. Observe that $\im (\overline{\partial}_{p+1}) = \Lambda^p \pi_s (\im (\partial_{p+1}))$ since $\Lambda^{p+1} \pi_s$ is surjective. It thus follows that the non-vanishing of $\ker(\pi_{s*})$ is equivalent to the existence of elements $k_p \in \ker (\partial_p) \setminus \im (\partial_{p+1})$ and $\lambda \in \im (\partial_{p+1})$ such that $k_p - \lambda \in \ker(\Lambda^p \pi_s)$.

Now observe that vector space isomorphism $\gtp \cong (\g \otimes t^s\C[t]) \oplus \gtps$ induces a vector space isomorphism $\Lambda^p (\gtp) \cong \bigoplus_{a+b=p} \Lambda^a(\g \otimes t^s\C[t]) \otimes \Lambda^b(\gtps)$. Under this identification one has $(\Lambda^p \pi_s)|_{\Lambda^p (\gtps)}=\id_{\Lambda^p (\gtps)}$ and $(\Lambda^p \pi_s) (\Lambda^{a+1}(\g \otimes t^s\C[t]) \otimes \Lambda^b(\gtps)) = 0$ for all $a+b+1=p$. Then
\[
\ker (\Lambda^p \pi_s) \subseteq F_{p(s-1)+1} \Lambda^p \subseteq F_{s} \Lambda^p.
\]
Since $k_p - \lambda \in \ker(\partial_p) \cap \ker(\Lambda^p \pi_s)$, $\ker(\Lambda^p \pi_s) \subseteq F_{s} \Lambda^p$ and $F_s \Lambda^p \cap \ker(\partial_{p}) = F_s \Lambda^p \cap \im(\partial_{p+1})$, it follows that $k_p - \lambda \in F_s \Lambda^p \cap \im (\partial_{p+1})$. Since $ \lambda \in \im (\partial_{p+1})$, this contradicts the assumption that $k_p \notin \im (\partial_{p+1})$. Thus, it must be the case that $\ker(\pi_{s*}) = \set{0}$.
\end{proof}

\begin{remark}
Given $s > 0$, set $\gt_s = \g \otimes \C[t]/\subgrp{t^s}$. Fishel, Grojnowski, and Teleman \cite[Remark 1.2]{FGT08} showed that $\h^n (\gt_s,\C) \cong \h^n (\g,\C)^{\otimes s}$ as algebras. Their result can be related to \cref{prop:4} by way of \cref{thm:gtcohomologyM}. Specifically, \cref{thm:gtcohomologyM} asserts for all $n \geq 0$ that
\[
\h^n (\gt_s,\C) \cong \bigoplus_{i+j=n} \h^i (\gt_s^+,\C)^\g \otimes \h^j (\g,\C).
\]
In particular, since $\h^1 (\g,\C) = \h^2 (\g,\C) = 0$ by the first and second Whitehead Lemmas, we get for $n = 2$ that $\h^2 (\gt_s,\C) \cong \h^2 (\gt_s^+,\C)^\g$.  By Proposition \ref{prop:4}, if $s > s_2$, then $\h^2 (\gt_s^+,\C) \cong \h^2 (\gtp,\C) \oplus \g^*$, which has no trivial $\g$-composition factors by \eqref{our.hn.gtp}. Then for $s \gg 0$ one has $\opH^2(\gt_s,\C) = 0$, which agrees with the results of \cite{FGT08}.
\end{remark}

\subsection{Second cohomology of \texorpdfstring{$\gIs$}{g otimes I/Is}} \label{ss:h2.gIs}

Recall that $I = \subgrp{(t-a)(t-b)} \unlhd \C[t]$, $a \neq b$, and $s>1$. The powers of $I$ define a decreasing algebra filtration on $\C[t]$, which induces a decreasing Lie algebra filtration $\gIs$. We denote the associated graded algebra for $\gIs$ by $\gr(\gIs)$. Then the cohomology rings $\Hbul(\gIs,\C)$ and $\Hbul(\gr(\gIs),\C)$ are naturally $\gt/(\gI) \cong \g \times \g$-modules; for further discussion see Appendix \ref{appendix}.

Let $s_2$ be as in \cref{prop:4}. This section is dedicated to proving the following result:

\begin{proposition} \label{prop:h2grgoIC}
If $s > s_2$, then $\h^2(\gr(\gIs),\C) \cong \h^2(\gIs,\C)$ as $\g\times\g$-modules.
\end{proposition}

\begin{proof}
As described in Appendix \ref{ss:ss}, there exists a homological-type spectral sequence
\begin{equation} \label{eq:grgIsspecseq}
E_{p,q}^1 \cong \h^{-(p+q)}(\gr(\gIs),\C)_{(p)} \Rightarrow \h^{-(p+q)}(\gIs,\C).
\end{equation}
By degree considerations, $E_{p,q}^1$ can be nonzero only if $p+q \leq 0 \leq 2p+q$. The locations of potentially nonzero terms in the $E^1$-page of \eqref{eq:grgIsspecseq} are illustrated in Figure \ref{fig:E1page}. As remarked in Appendix \ref{ss:ss}, \eqref{eq:grgIsspecseq} is a spectral sequence of $\g \times \g$-modules. Since $\g \times \g$ is a finite-dimensional semisimple Lie algebra, every short exact sequence of $\g \times \g$-modules splits. Then to prove the proposition, it suffices to show that $E_{p,q}^1 \cong E_{p,q}^\infty$ whenever $p+q=-2$.

\begin{figure}[htb]
\begin{tikzpicture}[scale=0.7]
\draw[step=1.0, gray, very thin] (0,-5.5) grid (5.5,0);
\draw[very thin, <->] (0,-5.5)--(0,0.5);
\draw (-0.2,.5) node{\scriptsize{$q$}};
\draw[black, very thin, <->] (-.5,0)--(5.7,0);
\draw (6,0) node{\scriptsize{$p$}};
\foreach \p in {1,2,3,4,5} \draw[white, fill=white] (\p,0) circle(5pt);
\foreach \p in {1,2,3,4,5} \draw (\p,0) node{\scriptsize{$0$}};
\foreach \p in {0,2,3,4,5} \draw[white, fill=white] (\p,-1) circle(5pt);
\foreach \p in {0,2,3,4,5} \draw (\p,-1) node{\scriptsize{$0$}};
\foreach \p in {0,3,4,5} \draw[white, fill=white] (\p,-2) circle(5pt);
\foreach \p in {0,3,4,5} \draw (\p,-2) node{\scriptsize{$0$}};
\foreach \p in {0,1,4,5} \draw[white, fill=white] (\p,-3) circle(5pt);
\foreach \p in {0,1,4,5} \draw (\p,-3) node{\scriptsize{$0$}};
\foreach \p in {0,1,5} \draw[white, fill=white] (\p,-4) circle(5pt);
\foreach \p in {0,1,5} \draw (\p,-4) node{\scriptsize{$0$}};
\foreach \p in {0,1,2} \draw[white, fill=white] (\p,-5) circle(5pt);
\foreach \p in {0,1,2} \draw (\p,-5) node{\scriptsize{$0$}};
\draw[fill=black] (0,0) circle(1.5pt);
\draw[fill=black] (1,-1) circle(1.5pt);
\draw[fill=black] (1,-2) circle(1.5pt);
\draw[fill=black] (2,-2) circle(1.5pt);
\draw[fill=black] (2,-3) circle(1.5pt);
\draw[fill=black] (3,-3) circle(1.5pt);
\draw[fill=black] (2,-4) circle(1.5pt);
\draw[fill=black] (3,-4) circle(1.5pt);
\draw[fill=black] (4,-4) circle(1.5pt);
\draw[fill=black] (3,-5) circle(1.5pt);
\draw[fill=black] (4,-5) circle(1.5pt);
\draw[fill=black] (5,-5) circle(1.5pt);
\end{tikzpicture}
\caption{Locations of potentially nonzero terms on the $E^1$-page of \eqref{eq:grgIsspecseq}.} \label{fig:E1page}
\end{figure}
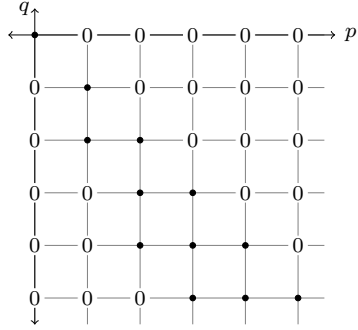

The polynomial grading on $\C[t]$ induces a grading on $\gtps$, hence also an additional non-negative grading on $\Hbul(\gtps,\C)$; denote the degree-$a$ component of $\Hbul(\gtps,\C)$ by $\Hbul(\gtps,\C)_{(a)}$. As in Section \ref{ss:hn.goI}, we refer to this additional non-negative grading as the $t$-degree. Now the discussion surrounding \eqref{eq:lem3} asserts the existence of a $\g \times \g$-module isomorphism
\begin{equation} \label{eq:ggisomorphism}
E_{p,q}^1 \cong \bigoplus_{\substack{m+n=-(p+q) \\ a+b = p}} \h^m(\gtps,\C)_{(a)} \boxtimes \h^n(\gtps,\C)_{(b)},
\end{equation}
where the symbol `$\boxtimes$' denotes the external tensor product of $\g$-modules. By the type of reasoning discussed in Appendix \ref{ss:ss}, one has
\begin{equation} \label{eq:degreevanishing}
\opH^m(\gtps,\C)_{(a)} = 0 \text{ unless $a \geq m$.}
\end{equation}

By direct observation, $\h^0(\gtps,\C)$ is isomorphic to $\C$ concentrated in $t$-degree $0$. Then \eqref{eq:ggisomorphism} and \eqref{eq:degreevanishing} imply that $E_{p,-p}^{1} = 0$ for all $p>0$. Next, we get from \eqref{eq:firstcohomology} that $\h^1 (\gtps,\C) \cong \g^*$ concentrated in $t$-degree $1$. Then $E_{1,-2}^1 \cong (\g^* \boxtimes \C) \oplus (\C \boxtimes \g^*)$, $E_{2,-4}^1 \cong \g^* \boxtimes \g^*$, and  $E_{p,-(p+1)}^{1} = 0$ for all $p>1$. Together with the condition that $E_{p,q}^r \neq 0$ only if $2p+q \geq 0$, this implies for each $r \geq 1$ that the endpoints of the sequences
\begin{align*}
&E_{(r+2),-(r+3)}^r \stackrel{d^r}{\longrightarrow} E_{2,-4}^r \stackrel{d^r}{\longrightarrow} E_{(2-r),(r-5)}^r, \quad \text{and} \\
&E_{(r+3),-(r+4)}^r \stackrel{d^r}{\longrightarrow} E_{3,-5}^r \stackrel{d^r}{\longrightarrow} E_{(3-r),(r-6)}^r
\end{align*}
are both zero, and hence that $E_{2,-4}^1 = E_{2,-4}^\infty$ and $E_{3,-5}^1 = E_{3,-5}^\infty$.

Now suppose that $s > s_2$. Then $\opH^2(\gtps,\C) \cong \opH^2(\gtp,\C) \oplus \g^*$ by \cref{prop:4}. It follows from \eqref{eq:degreescoadjoint} that the summand $\g^*$ of $\opH^2(\gtps,\C)$ is concentrated in $t$-degree $s$. By \cref{example:H2.gtp}, the summand $\opH^2(\gtp,\C)$ of $\opH^2(\gtps,\C)$ is concentrated in $t$-degree $3$ if $\ghat$ is of type $\wh{A}_1$, and is concentrated in $t$-degree $2$ otherwise. Combined with \eqref{eq:ggisomorphism} and the previous observation that $\opH^1(\gtps,\C)$ is concentrated in $t$-degree $1$, this implies that $E_{p,p-2}^1 \neq 0$ only if $p \in \set{2,3,s}$, and that $E_{s,s-2}^1 \cong (\g^* \boxtimes \C) \oplus (\C \boxtimes \g^*)$. Now to finish the proof we must show that $E_{s,s-2}^1 \cong E_{s,s-2}^\infty$.

By \cref{example:H2.gtp}, no $\g$-composition factor of $\opH^2(\gtp,\C)$ is isomorphic to $\g^*$. Then to show that $E_{s,s-2}^1 \cong E_{s,s-2}^\infty$, it suffices to show that $\g^* \boxtimes \C$ and $\C \boxtimes \g^*$ are $\g \times \g$-composition factors of $\opH^2(\gIs,\C)$. Consider the LHS spectral sequence for $\gI$ and its ideal $\gI^s$:
\begin{equation} \label{eq:LHSgIs}
E^{p,q}_2 \cong \h^p (\gIs, \h^q(\gI^s, \C)) \Rightarrow \h^{p+q} (\gI, \C).
\end{equation}
It is straightforward to check using \eqref{eq:firstcohomology} that the quotient map $\gI \rightarrow \gIs$ induces an isomorphism $\opH^1(\gIs,\C) \cong \opH^1(\gI,\C)$. Then the $5$-term exact sequence of low degree terms associated to \eqref{eq:LHSgIs} gives rise to an exact sequence
\[
0 \rightarrow \opH^0(\gIs,\opH^1(\g \otimes I^s,\C)) \rightarrow \opH^2(\gIs,\C) \rightarrow \opH^2(\gI,\C).
\]
Now a calculation like that in \eqref{eq:degreescoadjoint} shows that $\opH^0(\gIs,\opH^1(\g \otimes I^s,\C)) \cong (\g^* \boxtimes \C)\oplus (\C \boxtimes \g^*)$ as a $\g \times \g$-module (concentrated in $t$-degree $s$), and hence that $\g^* \boxtimes \C$ and $\C \boxtimes \g^*$ are $\g \times \g$-composition factors of $\opH^2(\gIs,\C)$.
\end{proof}

\subsection{Second cohomology of \texorpdfstring{$\gI$}{goI}} \label{ss:h2.goI}

In this section we give a more precise description of $\h^2(\gI,\C)$. Recall from the proof of \cref{prop:4} that the LHS spectral sequence \eqref{eq:LHSgIs} gives rise to an exact sequence of $\g \times \g$-modules
\begin{equation} \label{eq:5es.LHS.goI/goIs}
0 \rightarrow (\g^* \boxtimes \C) \oplus (\C \boxtimes \g^*) \stackrel{d}{\longrightarrow} \h^2 (\gIs, \C) \stackrel{\inf}{\longrightarrow} \h^2 (\gI, \C).
\end{equation}
(The existence of this exact sequence is not dependent on the assumption that $s > s_2$.) The last arrow in \eqref{eq:5es.LHS.goI/goIs} is the inflation map in cohomology induced by the quotient $\gI \rightarrow \gIs$. Since $\im(d) = \ker(\inf)$, one has $\im(\inf) \cong \opH^2(\gIs,\C)/\ker(\inf)$ and $\ker(\inf) \cong (\g^* \boxtimes \C) \oplus (\C \boxtimes \g^*)$.

Now assume that $s > s_2$. Given $w \in W_a^1$, write $w \cdot 0 = \lambda_w - d_w \delta$ for some $\lambda_w \in P^+$ and some $d_w \geq 0$ as in Section \ref{ss:Hn.gtp}. Then it follows from the proof of \cref{prop:h2grgoIC} (cf.\ also \cref{example:H2.gtp}) that $\g^* \boxtimes  \g^*$, $V (\lambda_w)^* \boxtimes \C$, and $\C \boxtimes V (\lambda_w)^*$ for $w \in W_a^1$ with $\ell(w) = 2$ are $\g \times \g$-composition factors of $\h^2 (\gI, \C)$, concentrated in $t$-degrees $2$, $d_w$, and $d_w$, respectively.

\begin{proposition}{\label{prop0}}
If $\lambda,\mu \in P^+$ and $\lambda \neq 0$, then setting $c = b-a$,
\[
\Hom_{\g \times \g}( V(\lambda)^* \boxtimes V(\mu)^*, \h^2(\gI,\C)) \cong \Hom_{\g}(V(\lambda)^*, \h^2(\gtp, \ev^*_c V(\mu)^*)).
\]
Similarly, if $\mu \neq 0$, then setting $c = a-b$,
\[
\Hom_{\g \times \g}( V(\lambda)^* \boxtimes V(\mu)^*, \h^2(\gI,\C)) \cong \Hom_{\g}(V(\mu)^*, \h^2(\gtp, \ev^*_c V(\lambda)^*)).
\]
\end{proposition}

\begin{proof}
We establish the first isomorphism; the second isomorphism holds by a symmetric argument. Consider the LHS spectral sequence
\[ 
E_2^{p,q} \cong \h^p(\gt/\gI,\h^q(\gI,\ev_a^* V(\lambda) \otimes \ev_b^* V(\mu))) \Rightarrow \h^{p+q}(\gt,\ev_a^* V(\lambda) \otimes \ev_b^* V(\mu)).
\]
Since $\gt/\gI \cong \g \times \g$ by the Chinese Remainder Theorem, and since $\gI$ acts trivially on $\ev_a^* V(\lambda) \otimes \ev_b^* V(\mu)$, it follows from \cite[\S24]{CE48} that 
\[
E_2^{p,q} \cong \h^p (\g \times \g, \C) \otimes \h^0 (\g \times \g, \h^q (\gI, \C) \otimes (V(\lambda) \boxtimes V(\mu))).
\]
Since $\lambda \neq 0$, the $\g \times \g$-module $V(\lambda) \boxtimes V(\mu)$ is nontrivial and $E_2^{p,0} = 0$ for all $p \geq 0$. Since $\g$ is simple, the first and second Whitehead Lemmas imply that $E^{1,q}_2 = E^{2,q}_2 = 0$ for all $q \geq 0$. Thus
\begin{equation} \label{h2.1}
\begin{split}
\h^2(\gt,\ev_a^* V(\lambda) \otimes \ev_b^* V(\mu)) \cong E^{0,2}_2
&\cong \h^0 (\g \times \g, \h^2 (\gI, \C) \otimes (V(\lambda) \boxtimes V(\mu))) \\
&\cong \Hom_{\g \times \g}(V(\lambda)^* \boxtimes V(\mu)^*, \h^2 (\gI, \C)).
\end{split}
\end{equation}

Now consider the LHS spectral sequence for the Lie algebra $\gt$ and its ideal $\g \otimes \subgrp{(t-a)}$:
\begin{multline*}
E_2^{p,q} = \opH^p(\gt/\g \otimes \subgrp{(t-a)},\opH^q(\g \otimes \subgrp{(t-a)},\ev_a^* V(\lambda) \otimes \ev_b^* V(\mu))) \\
\Rightarrow \opH^{p+q}(\gt,\ev_a^* V(\lambda) \otimes \ev_b^* V(\mu)).
\end{multline*}
By similar reasoning as before, the $E_2$-page of this spectral sequence can be written in the form
\[ 
E_2^{p,q}
\cong \h^p(\g,\C) \otimes \h^0 (\g, \h^q (\g \otimes \subgrp{(t-a)}, \ev_b^* V(\mu)) \otimes V(\lambda)).
\]
Observe that $\g \otimes \subgrp{(t-a)} \cong \gtp$ via the linear change of variable $t \mapsto t+a$. Then
\[ 
E_2^{p,q}
\cong \h^p(\g,\C) \otimes \h^0 (\g, \h^q (\gtp, \ev_{b-a}^* V(\mu)) \otimes V(\lambda)).
\]
Since $\lambda \neq 0$, the $\g$-module $V(\lambda)$ is nontrivial and $E_2^{p,0} = 0$ for all $p \geq 0$. Applying the first and second Whitehead Lemmas again, one gets $E^{1,q}_2 = E^{2,q}_2 = 0$ for all $q \geq 0$. Thus, setting $c=b-a$,
\begin{equation} \label{h2.1b}
\begin{split}
\h^2 (\gt, \ev_a^* V(\lambda) \otimes \ev_b^* V(\mu)) \cong E^{0,2}_2 
&\cong \h^0 (\g, \h^2 (\gtp, \ev^*_c V(\mu)) \otimes V(\lambda)) \\
&\cong \Hom_{\g}(V(\lambda)^*, \h^2 (\gtp, \ev^*_c V(\mu))).
\end{split}
\end{equation}
Then the proposition follows by comparing \eqref{h2.1} and \eqref{h2.1b}.
\end{proof}

\begin{corollary} \label{lemma2}
Let $0 \neq \lambda \in P^+$. The $\g \times \g$-modules $V(\lambda)^* \boxtimes \C$ and $\C \boxtimes V(\lambda)^*$ are composition factors of $\h^2 (\gI, \C)$ if and only if $\lambda = \lambda_{w}$ for some $w \in W_a^1$ with $\ell(w) = 2$. If they occur as composition factors, then they do so with multiplicity $1$.
\end{corollary}

\begin{proof}
By \cref{prop0},
\[
\Hom_{\g \times \g} (V(\lambda)^* \boxtimes \C, \h^2 (\gI, \C)) \cong \Hom_{\g} (V(\lambda)^*, \h^2(\gtp, \C)),
\]
and by \eqref{our.hn.gtp},
\[
\Hom_{\g} (V(\lambda)^*, \h^2(\gtp, \C)) \cong \bigoplus_{\substack{w \in W_a^1 \\ \ell(w)=2}} \Hom_{\g}(V(\lambda)^* , V(\lambda_{w})^*).
\]
Since $V(\lambda)$ and $V(\lambda_{w})$ are simple $\g$-modules, we have
\[
\dim \Hom_{\g} (V(\lambda)^*, V(\lambda_{w})^*) =
\begin{cases}
1 & \text{if $\lambda = \lambda_{w}$}, \\
0 & \text{otherwise.}
\end{cases}
\]
This proves the statement for $V(\lambda^*) \boxtimes \C$, and the other case follows by a symmetric argument.
\end{proof}

The following result determines the multiplicity of the trivial $\g\times \g$-module in $\h^2 (\gI, \C)$. 

\begin{proposition} \label{prop3}
The multiplicity of $\C \boxtimes \C$ as $\g \times \g$-composition factor of $\h^2 (\gI, \C)$ is $1$.
\end{proposition}

\begin{proof}
Consider the LHS spectral sequence for the Lie algebra $\gt$ and its ideal $\gI$:
\[
E_2^{p,q} \cong \h^p (\gt/\gI, \h^q (\gI, \C)) \Rightarrow \h^{p+q} (\gt, \C).
\]
Observe that, as in the proof of \cref{prop0}, $E^{p,q}_2 \cong \h^p(\g \times \g,\C) \otimes \h^0(\g \times \g, \h^q(\gI,\C))$. Also recall that $\h^n(\gt,\C) \cong \h^n(\g,\C)$ by \cref{theorem:gtpginvariants}.

Since $\g$ is a finite-dimensional simple Lie algebra, $\opH^3(\g,\C) \cong \C$. Then the K\"{u}nneth formula implies that $E^{3,0}_2 = \h^3(\g \times \g, \C) \cong \C^2$. The K\"{u}nneth formula, together with the first and second Whitehead Lemmas, also implies that $E^{1,q}_2 = E^{2,q}_2 = 0$ for all $q \geq 0$. Then $E_2^{0,2} = E_3^{0,2}$ and $E_2^{3,0}=E_3^{3,0}$. Since $E^{3,0}_\infty = \coker(d^{0,2}_3 : E^{0,2}_2 \rightarrow E^{3,0}_2)$ is a quotient of the two-dimensional space $E^{3,0}_2$, and since $E^{3,0}_\infty$ is a section of the one-dimensional space $\opH^3(\gt, \C) \cong \opH^3(\g,\C) \cong \C$, the map $d^{0,2}_3: E^{0,2}_2 \rightarrow E^{3,0}_2$ cannot be trivial. This proves that $\dim \Hom_{\g \times \g} (\C \boxtimes \C, \h^2 (\gI, \C)) \geq 1$.

Now to prove the reverse inequality, consider the following LHS spectral sequence:
\[
E_2^{p,q} \cong \h^p (\g \otimes \subgrp{t-a}/\gI, \h^q(\gI,\C))
\Rightarrow \h^{p+q} (\g \otimes \subgrp{t-a},\C) .
\]
Observe that $E^{p,q}_2 \cong \h^p (\g, \C) \otimes \h^0 (\g, \h^q (\gI, \C))$. Here, $\h^0 (\g, \h^q (\gI, \C))$ identifies with $\h^q (\gI, \C)^{\C \times \g}$ when we view $\h^q (\gI, \C)$ as a $\g \times \g$-module as before; in other words, this $\h^{0}$ ``detects'' composition factors of the form $\C \boxtimes V(\nu)$. Also observe that $\h^n (\g \otimes \langle t-a \rangle, \C) \cong \h^n (\gtp, \C)$, and recall from \eqref{our.hn.gtp} that $\h^n (\gtp, \C) \cong \bigoplus_{w\in W_a^1, \ell (w) = n} V(\lambda_{w})^*$.

Since $\g$ is a finite-dimensional, simple Lie algebra, $E_2^{3,0} \cong \h^3 (\g, \C) =\C$, and the first and second Whitehead Lemmas imply that $E^{1,q}_2 = E^{2,q}_2 = 0$ for all $q\geq 0$. Thus, similar to the analysis above, $\h^2 (\gtp, \C) = E^{2,0}_\infty = \ker(d^{0,2}_3: E^{0,2}_2 \rightarrow E^{3,0}_2)$, implying that 
\[
\dim \h^{2}(\gI, \C)^{\C \times \g} = \dim E_{2}^{0,2} = \dim(\ker d^{0,2}_3) + \dim (\im d^{0,2}_3) \leq \sum_{\substack{w \in W_a^1 \\ \ell(w) = 2}} \dim V(\lambda_{w})^* + 1.
\]
From Corollary \ref{lemma2}, we know that the $\C \boxtimes V (\lambda_{w})^*$ are composition factors of $\h^2 (\gI, \C)$, so $\dim E_{2}^{0,2} \geq \sum_{w \in W_a^1, \ell(w) = 2} \dim V(\lambda_{w})^*$. Our dimension estimate leaves room for at most one composition factor $\C \boxtimes \C$, proving that $\dim \Hom_{\g \times \g} (\C \boxtimes \C, \h^2 (\gI, \C)) \leq 1$.
\end{proof}

\begin{example}{\label{eg:h2diff}}
Recall that using the K\"unneth formula (see \eqref{eq:lem3}), 
\[
\h^2 (\gr(\gI), \C) \cong (\h^2(\gtp, \C) \boxtimes \C) \oplus (\g^* \boxtimes \g^*) \oplus (\C \boxtimes \h^2 (\gtp, \C)).
\]
Thus, by \cref{example:H2.gtp}, $\h^2 (\gr(\gI), \C)$ does not contain any trivial $\g \times \g$-composition factor. By \cref{prop3}, $(\C \boxtimes \C)$ is a composition factor of $\h^2 (\gI, \C)$ with multiplicity 1. This explicitly shows a difference between $\h^2 (\gr(\gI), \C)$ and $\h^2 (\gI, \C)$ (also see Section \ref{section:non-conv.ss}).
\end{example}


\section{Calculation of \texorpdfstring{$\opH^{2}(\fsl_{2}\otimes I, \C)$}{H2(sl2 otimes I,C)}} \label{ss:h2.sl2}

\subsection{}

Recall that $I = \subgrp{f} \unlhd \C[t]$, where $f = (t-a)(t-b) \in \C[t]$ and $a \neq b$. The goal of this section is to prove the following theorem, which explicitly describes the structure of $\h^2 (\fsl_2 \otimes I, \C)$ as an $\fsl_2 \times \fsl_2$-module.

\begin{theorem} \label{thm:H^2.sl2oI}
$\h^2 (\fsl_2 \otimes I, \C)$ is isomorphic as a $\fsl_2 \times \fsl_2$-module to
\[
[\C \boxtimes \C] \oplus [V(2) \boxtimes V(2)] \oplus [V(4) \boxtimes \C] \oplus [\C \boxtimes V(4)].
\]
\end{theorem}

From \cref{fin.comp.fac}, it follows that the only possible composition factors of $\h^2 (\fsl_2 \otimes I, \C)$ are of the form $V(\lambda) \boxtimes V(\mu)$ with $\lambda + \mu \in \{0,2,4\}$. From \cref{prop3}, it follows that $\h^2 (\fsl_2 \otimes I, \C)$ has $\C \boxtimes \C$ as a composition factor with multiplicity 1. From \cref{lemma2}, it follows that $\C \boxtimes V(2)$ and $V(2) \boxtimes \C$ are not composition factors of $\h^2 (\fsl_2 \otimes I, \C)$. From \cref{lemma2} and \cref{example:H2.gtp}, it follows that $\C \boxtimes V(4)$ and $V(4) \boxtimes \C$ are composition factors of $\h^2 (\fsl_2 \otimes I, \C)$ with multiplicity 1.

The remaining cases include determining the multiplicities of $V(1)\boxtimes V(1)$, $V(2)\boxtimes V(2)$, $V(1)\boxtimes V(3)$, 
and $V(3)\boxtimes V(1)$. 

\subsection{}

Throughout the rest of this section, set $\g = \fsl_2$ and fix a Chevalley basis $\{y,h,x \}$ for $\g$ such that $y \in \g_{-2}$, $x \in \g_2$, and $h = [x,y] \in \fh$.

\begin{proposition} \label{h2.sl2.1.3} Let $a \neq 0$. Then
\begin{itemize}
\item[(a)] $\Hom_{\fsl_2} (V(3), \h_2 (\fsl_2 [t]^+, \ev_a^*V(1))) = 0$.
\item[(b)] $\Hom_{\fsl_2} (V(1), \h_2 (\fsl_2 [t]^+, \ev_a^*V(1))) = 0$.
\end{itemize} 
\end{proposition}

\begin{proof} Fix a basis $\{ v_{-1}, v_1 \} \subset \ev_a^*V(1)$ such that
\[
\begin{array}{c c c}
x \cdot v_1 = 0  & & x \cdot v_{-1} = v_1 \\
h \cdot v_1 = v_1  & & h \cdot v_{-1} = -v_{-1} \\
y \cdot v_1 = v_{-1}  & & y \cdot v_{-1} = 0.
\end{array}
\]
There exists a $\g$-module isomorphism
\begin{equation} \label{iso.ext2.sl2}
\Lambda^2(\gtp) \otimes \ev_a^*V(1) 
\cong  \bigoplus_{i \geq 1} \Lambda^2 (\g \otimes \C t^i) \otimes V(1)
\ \oplus \ \bigoplus_{1 \leq i < j} (\g \otimes \C t^i) \otimes (\g \otimes \C t^j) \otimes V(1).
\end{equation}
For part (a), using this isomorphism, the weight-3 subspace of $\Lambda^2(\gtp) \otimes \ev_a^*V(1)$ is spanned by 
\[
\{ (x \otimes t^i) \wedge (x \otimes t^j) \otimes v_{-1}, 
\ (x \otimes t^\ell) \wedge (h \otimes t^m) \otimes v_1 \, : 
\, 1 \leq i < j; \ 1 \leq \ell, m \}.
\]
Thus, non-zero elements of weight 3 in the kernel of 
\[
\begin{array}{rcl}
\partial_2 : \Lambda^2(\gtp) \otimes \ev_a^*V(1) & \longrightarrow & \gtp \otimes \ev_a^*V(1) \\
x \wedge y \otimes v & \longmapsto & -[x,y] \otimes v + y \otimes xv - x \otimes yv
\end{array}
\]
are scalar multiples of
\[
k_{i,j} = (x \otimes t^i) \wedge (x \otimes t^j) \otimes v_{-1} - ( (x \otimes t^i) \wedge (h \otimes t^j) \otimes v_1 + (h \otimes t^i) \wedge (x \otimes t^j) \otimes v_1), \quad 1 \leq i < j.
\]
Since $x \cdot k_{i,j} \neq 0$, it follows that $k_{i,j}$ does not represent a highest-weight vector of weight 3 in $\h_2 (\gtp, \ev_a^*V(1))$ for any $1 \leq i < j$. Thus,  $\Hom_\g (V(3), \h_2 (\gtp, \ev_a^*V(1))) = 0$.

For part (b), using the isomorphism \eqref{iso.ext2.sl2}, the weight-1 subspace of $\Lambda^2(\gtp) \otimes \ev_a^*V(1)$ is spanned by
\[
\{ (x \otimes t^i) \wedge (h \otimes t^j) \otimes v_{-1}, \,
(x \otimes t^i) \wedge (y \otimes t^j) \otimes v_1, \,
(h \otimes t^\ell) \wedge (h \otimes t^m) \otimes v_1 \, 
: \, 1 \leq i, j; \, 1 \leq \ell < m\}.
\]
Therefore, non-zero elements of weight 1 in the kernel of
\[
\partial_2 : \Lambda^2(\gtp) \otimes \ev_a^*V(1) \rightarrow \gtp \otimes \ev_a^*V(1)
\]
are scalar multiples of
\begin{align*}
k_{i,j}
{ }&{ } = (x \otimes t^i) \wedge (h \otimes t^j) \otimes v_{-1} 
+ (h \otimes t^i) \wedge (x \otimes t^j) \otimes v_{-1} \\
{ }&{ }+ (x \otimes t^i) \wedge (y \otimes t^j) \otimes v_1
+ (y \otimes t^i) \wedge (x \otimes t^j) \otimes v_1
- (h \otimes t^i) \wedge (h \otimes t^j) \otimes v_1, \ 1 \leq i < j.
\end{align*}
Since $x \cdot k_{i,j} \neq 0$, it follows that $k_{i,j}$ does not represent a highest-weight vector in $\h_2 (\gtp, \ev_a^*V(1))$ for any $1 \leq i <j$. Thus, $\Hom_\g (V(1), \h_2 (\gtp, \ev_a^*V(1))) = 0$.
\end{proof}

The following corollary settles the question about the multiplicities of the composition factors $V(1)\boxtimes V(1)$, $V(1)\boxtimes V(3)$, 
and $V(3)\boxtimes V(1)$. 
\vskip 1cm 

\begin{corollary} \label{cor:h2.sl2.1.3} {\ }
\begin{itemize}
\item[(a)] $\Hom_{\fsl_2 \times \fsl_2} (V(3) \boxtimes V(1), \h^2 (\fsl_2 \otimes I, \C)) = 0$.
\item[(b)] $\Hom_{\fsl_2 \times \fsl_2} (V(1) \boxtimes V(3), \h^2 (\fsl_2 \otimes I, \C)) = 0$.
\item[(c)] $\Hom_{\fsl_2 \times \fsl_2} (V(1) \boxtimes V(1), \h^2 (\fsl_2 \otimes I, \C)) = 0$.
\end{itemize} 
\end{corollary}

\begin{proof}
We will prove part (a); parts (b) and (c) follow by similar arguments. From \cref{prop0}, there is an isomorphism 
\[
\Hom_{\g \times \g} (V(3) \boxtimes V(1), \h^2 (\gI, \C)) \cong \Hom_{\g} (V(3), \h^2 (\gtp, \ev_c^* V(1)^*)). 
\]
By Lemma \ref{lem:duality} and the fact that finite-dimensional $\fsl_2$-modules are self-dual, there is an isomorphism of $\fsl_2$-modules $\h^2 (\gtp, \ev_c^* V(1)^*) \cong \h_2 (\gtp, \ev_c^* V(1))$. 
Therefore, one can use \cref{h2.sl2.1.3} to show that $\Hom_{\g} (V(3), \h^2 (\gtp, \ev_c^* V(1)))=0$.
\end{proof}

\subsection{}

Up to this point, we know that there is an $\fsl_2 \times \fsl_2$-module isomorphism
\[
\h^2 (\fsl_2 \otimes I, \C)
\cong
\C \boxtimes \C \oplus (V(2) \boxtimes V(2))^m \oplus V(4) \boxtimes \C \oplus \C \boxtimes V(4),
\] 
for some $m \ge 0$. In order to explicitly compute $m$, rewrite $\h^2(\fsl_2 \otimes I, \C)$ as a module for the diagonal subalgebra $\fsl_2$ of $\fsl_2 \times \fsl_2$,
\[
\h^2 (\fsl_2 \otimes I, \C) \cong \C^{m+1} \oplus V(2)^m \oplus V(4)^{m+2} ,
\]
and recall from \cref{prop:HC1} that $\dim \Hom_{\fsl_2} (\C , \h^2 (\fsl_2 \otimes I, \C)) = 2$.  This implies that $m=1$, and finishes the proof of Theorem~\ref{thm:H^2.sl2oI}.


\appendix

\section{Spectral sequences} \label{appendix}

\subsection{A convergent spectral sequence} \label{ss:ss}

In this section we construct the spectral sequence that is used in Section~\ref{ss:h2.gIs} to compute the cohomology group $\h^2(\gIs,\C)$.

Fix distinct elements $a_1,\dots,a_k \in \C$, and set $I = \subgrp{(t-a_1) \cdots (t-a_k)} \unlhd \C[t]$. The powers of $I$ define a decreasing filtration $\C[t] \supset I \supset I^2 \supset \cdots$ on $\C[t]$, which induces a decreasing Lie algebra filtration $\gt \supset \gI \supset \gI^2 \supset \cdots$ on $\gt$ and on its subalgebra $\gI$. Given $s > 1$, one obtains an induced decreasing filtration on the Lie algebra $\gIs \cong1 (\gI)/(\g \otimes I^s)$. The associated graded Lie algebras are then defined by
\begin{align*}
\gr(\gt) &= \bigoplus_{n \geq 0} (\gI^n)/(\gI^{n+1}), \\
\gr(\gI) &= \bigoplus_{n \geq 1} (\gI^n)/(\gI^{n+1}), \text{ and} \\
\gr(\gIs) &= \bigoplus_{1 \leq n < s} (\g \otimes I^n/I^s)/(\g \otimes I^{n+1}/I^s).
\end{align*}
The adjoint action of $\gt$ on itself induces an action of $\gt$ on each of these associated graded Lie algebras. This action factors through the quotient Lie algebra $\gt/(\gI) \cong \g \otimes \C[t]/I$. Since $I = \subgrp{(t-a_1) \cdots (t-a_k)}$ and the $a_i$ are distinct, it follows from the Chinese Remainder Theorem that $\g[t]/(\gI) \cong \gk$ as Lie algebras. More generally, one obtains Lie algebra isomorphisms
\begin{align}
\gr(\gt) &\cong \gt^{\oplus k}, \label{eq:grgtalgebraiso} \\
\gr(\gI) &\cong (\gtp)^{\oplus k}, \text{ and} \label{eq:grgIalgebraiso} \\
\gr(\gIs) &\cong (\gtps)^{\oplus k}, \label{eq:grgIsalgebraiso}
\end{align}
where $\gtps := \g \otimes (t\C[t]/t^s \C[t])$ as in Section \ref{ss:ldcahogtps}. Considering $\gt^{\oplus k} = (\g \otimes \C[t])^{\oplus k}$ as a graded space via the polynomial grading on $\C[t]$, and then considering the induced gradings on $(\gtp)^{\oplus k}$ and $(\gtps)^{\oplus k}$, the previous three isomorphisms are isomorphisms of graded Lie algebras. Restricting to the $0$-graded parts of each graded Lie algebra, (\ref{eq:grgtalgebraiso}--\ref{eq:grgIsalgebraiso}) are $\gk$-module isomorphisms.

The decreasing filtration on $\gIs$ induces a decreasing filtration on its exterior algebra:
\begin{equation}{\label{eq:filt.gIs}}
F^{j} \Lambda^n( \gIs )
:= \sum_{\substack{j \leq j_1+j_2+\cdots+j_n \\ 0 < j_1, \ldots, j_n < s}} (\gI^{j_1}/I^s) \wedge (\gI^{j_2}/I^s) \wedge \cdots \wedge (\gI^{j_n}/I^s).
\end{equation}
This filtration is bounded above because $F^j \Lambda^n (\gIs) = \Lambda^n (\gIs)$ if $j \leq n$, and is bounded below because $F^j \Lambda^n (\gIs) = 0$ if $j > n(s-1)$. This filtration is also compatible with the differential on the Koszul resolution $U(\gIs) \otimes \Lambda^\bullet(\gIs)$. Now recall that $\Hbul(\gIs,\C)$ can be computed as the cohomology of the dual Koszul complex $C^\bullet := \Hom_\C(\Lambda^\bullet(\gIs),\C)$. Define an increasing filtration $F^\bullet$ on $C^\bullet$ by
\[
F^j C^n = \Hom_\C(\Lambda^n(\gIs)/F^{j+1} \Lambda^n(\gIs),\C).
\]
Then $F^\bullet$ is compatible with the Koszul differential on $C^\bullet$, and is bounded above and below since the original filtration on $\Lambda^\bullet(\gIs)$ was bounded above and below.

The exterior algebra $\Lambda^\bullet(\gr(\gIs))$ inherits an additional non-negative grading from the grading on $\gr(\gIs)$; denote the $j$-th graded component of $\Lambda^\bullet(\gr(\gIs))$ by $\Lambda^\bullet(\gr (\gIs))_{(j)}$. This grading passes to an additional non-negative grading on the cohomology ring $\h^\bullet (\gr(\gIs),\C)$; denote its $p$-th graded component $\Hbul(\gr(\gIs),\C)_{(p)}$. Now one can check that
\begin{equation} \label{eq:grC}
F^j C^n / F^{j-1} C^n \cong \Hom_\C(\Lambda^n(\gr (\gIs))_{(j)},\C),
\end{equation}
hence the associated graded complex $\gr(C^\bullet)$ arising from the filtration $F^\bullet$ on $C^\bullet$ identifies with the dual Koszul complex for $\gr(\gIs)$. In particular, $\opH^n(\gr(C^\bullet)) \cong \opH^n(\gr(\gIs),\C)$.

Write $\delta$ for the Koszul differential on $C^\bullet$. Set $C_n = C^{-n}$, and set $F_j C_n = F^j C^{-n}$. Then $(C_\bullet,\delta)$ is a chain complex with differential $\delta$ of degree $-1$, and $F_\bullet$ is an increasing filtration on $C_\bullet$. Since the filtration $F_\bullet C_\bullet$ is bounded below and above, there exists by \cite[Theorem 5.5.1]{Weibel:1994} a convergent spectral sequence of homological type
\begin{equation} \label{eq:homologicalspecseq}
E_{p,q}^1 = \opH_{p+q}(F_p C_\bullet/F_{p-1}C_\bullet) \Rightarrow \opH_{p+q}(C_\bullet).
\end{equation}
By \eqref{eq:grC}, we can rewrite \eqref{eq:homologicalspecseq} as
\begin{equation} \label{eq:may.ss}
E_{p,q}^1 = \opH^{-(p+q)}(\gr(\gIs),\C)_{(p)} \Rightarrow \opH^{-(p+q)}(\gIs,\C).
\end{equation}
Observe that $E_{p,q}^1 = 0$ unless $p \geq 0$ and $p+q \leq 0 \leq 2p + q$; the condition $2p+q \geq 0$ comes from the fact that $\Lambda^n(\gr(\gIs))_{(p)} = 0$ and hence $\opH^n(\gr(\gIs),\C)_{(p)} = 0$ unless $p \geq n$. Since $\gIs$ and $\gr(\gIs)$ are finite-dimensional algebras, it follows that all of the individual terms appearing in \eqref{eq:may.ss} are finite-dimensional, and also for each $n$ that the sum $\bigoplus_{p+q=n} E_{p,q}^\infty$ involves only finitely many nonzero summands.

The adjoint action of $\gt$ on itself induces a filtration-preserving action of $\gt$ on $C_\bullet$. In particular, the ideal $\gI$ of $\gt$ maps $F_j C_n$ into $F_{j-1} C_n$. Then it follows that \eqref{eq:may.ss} is a spectral sequence of $\gt/(\gI) \cong \gk$-modules.

As stated in \eqref{eq:grgIsalgebraiso}, there exists an isomorphism of graded Lie algebras $\gr(\gIs) \cong (\gtps)^{\oplus k}$. Then it follows from the K\"{u}nneth formula (cf.\ \cite[Exercise 7.3.8]{Weibel:1994}) that $\Hbul(\gr(\gIs),\C) \cong \Hbul(\gtps,\C)^{\otimes k}$ as graded-commutative algebras. In particular, there are $\gk$-module isomorphisms
\begin{equation} \label{eq:lem3}
\h^n(\gr(\gIs),\C) \cong \bigoplus_{n_1 + \cdots + n_k = n} \h^{n_1}(\gtps,\C) \boxtimes \cdots \boxtimes  \h^{n_k}(\gtps,\C),
\end{equation}
where the symbol `$\boxtimes$' denotes the external tensor product of $\g$-modules. These isomorphisms are compatible with the additional non-negative gradings on $\Hbul(\gr(\gIs),\C)$ and $\Hbul(\gtps,\C)$.

Since $\g$ is a finite-dimensional simple Lie algebra, every finite-dimensional representation of $\gk$ is completely reducible. Then the spectral sequence \eqref{eq:may.ss} implies that $\opH^n(\gIs,\C)$ is isomorphic to a $\gk$-module subquotient of the expression on the right-hand side of \eqref{eq:lem3}.

\subsection{A spectral sequence that converges to the wrong limit} \label{section:non-conv.ss}

In this section we discuss what happens when one replaces the finite-dimensional Lie algebra $\gIs$ in the previous section with the infinite-dimensional Lie algebra $\gI$. In this case, one again obtains a decreasing filtration $F^\bullet$ on the exterior algebra for $\Lambda^\bullet(\gI)$, defined by
\[
F^j \Lambda^n(\gI) = \sum_{j_1+\cdots+j_n \geq j} (\gI^{j_1}) \wedge \cdots \wedge (\gI^{j_n}),
\]
though the filtration is no longer bounded below. Still, there is an induced increasing filtration on the dual Koszul complex $C^\bullet := \Hom_\C(\Lambda^\bullet(\gI),\C)$, defined by
\[
F^j C^n = \Hom_\C(\Lambda^n(\gI)/F^{j+1} \Lambda^n(\gI),\C).
\]
This filtration is bounded below but is not bounded above. In particular, the filtration $F^\bullet$ on $C^\bullet$ is not exhaustive; i.e., $\bigcup_{j \geq 0} F^j C^n \subsetneq C^n$. Nevertheless, the filtration gives rise as in Section \ref{ss:ss} to a fourth-quadrant homological-type spectral sequence with
\begin{equation} \label{eq:grgIspecseq}
E_{p,q}^1 \cong \opH^{-(p+q)}(\gr(\gI),\C)_{(p)}.
\end{equation}

\begin{lemma}{\label{lem:ssgI}}
The spectral sequence \eqref{eq:grgIspecseq} stabilizes at the $E^1$-page; i.e., $E^1 \cong E^r$ for all $r \geq 1$.
\end{lemma}

\begin{proof}
As in Section \ref{ss:ss}, \eqref{eq:grgIspecseq} is a spectral sequence of $\gt/(\gI) \cong \gk$-modules. Next, the Lie algebra isomorphism $\gr(\gI) \cong (\gtp)^{\oplus k}$ gives rise to a $\gk$-module isomorphism
\begin{equation} \label{E:HgrgxI}
\opH^\bullet(\gr (\gI),\C) \cong \opH^\bullet(\gtp,\C)^{\boxtimes k} = \opH^\bullet(\gtp,\C) \boxtimes \cdots \boxtimes \opH^\bullet(\gtp,\C) \quad \text{($k$ factors)}.
\end{equation}
By \eqref{our.hn.gtp}, $\opH^\bullet(\gtp,\C)$ is completely reducible and multiplicity free as a $\g$-module. Since the differential $d^1$ on $E^1$ is a $ \gk$-module homomorphism, this implies that $d^1 \equiv 0$, so that $E^1 \cong E^2$. Now an evident induction argument shows that $d^i \equiv 0$ and $E^i \cong E^{i+1}$ for all $i \geq 1$.
\end{proof}

Recall that the adjoint action of a Lie algebra $\fa$ on itself always induces a trivial $\fa$-action on the Lie algebra cohomology ring $\Hbul(\fa,\C)$ (see \cite[\S 3.1.2]{Kumar:2002a}), so the adjoint action of $\gt$ on $\Hbul(\gI,\C)$ factors through the quotient $\gt/(\gI) \cong \gk$. 

\begin{lemma}
For $k>1$, $\Hbul(\gI,\C)$ is not isomorphic as a $\gk$-module to $\Hbul(\gr (\gI),\C)$.
\end{lemma}

\begin{proof}
Suppose $\Hbul(\gI,\C) \cong \Hbul(\gr (\gI),\C)$ as $\gk$-modules. Then by \eqref{E:HgrgxI}, $\Hbul(\gI,\C) \cong \Hbul(\gtp,\C)^{\boxtimes k}$ as $\gk$-modules. Now consider the LHS spectral sequence for the Lie algebra $\gt$ and its ideal $\gI$:
\[
E_2^{i,j} = \opH^i(\gt/(\gI),\opH^j(\gI,\C)) \Rightarrow \opH^{i+j}(\gt,\C).
\]
One could then rewrite the $E_2$-page as
\begin{align*}
E_2^{i,j} 
&= \bigoplus_{j_1 + \dots + j_k = j} \opH^i(\gk, \opH^{j_{1}}(\gtp,\C) \boxtimes \dots  \boxtimes \opH^{j_{k}}(\gtp,\C)) \\
&\cong \bigoplus_{\substack{i_1 + \dots + i_k = i \\ j_1 + \dots + j_k = j}} \opH^{i_{1}}(\g, \opH^{j_{1}}(\gtp,\C)) \otimes \dots \otimes \opH^{i_{k}}(\g, \opH^{j_{k}}(\gtp,\C))  \\
&\cong \bigoplus_{\substack{i_1 + \dots + i_k = i \\ j_1 + \dots + j_k = j}} [\opH^{i_{1}}(\g, \C) \otimes \opH^{j_{1}}(\gtp,\C)^\g] \otimes \dots \otimes [\opH^{i_{k}}(\g,\C) \otimes  \opH^{j_{k}}(\gtp,\C)^\g].
\end{align*}
By Theorem \ref{theorem:gtpginvariants}, $\Hbul(\gtp,\C)^\g = \opH^0(\gtp,\C) = \C$. Then one would have $E_2^{i,j} = 0$ for all $j > 0$, and the spectral sequence would collapse to yield isomorphisms
\[
\Hbul(\gt,\C) \cong E_2^{\bullet,0} = \Hbul(\gt/\gI,\C) \cong \Hbul(\gk, \C) \cong \Hbul(\g,\C)^{\otimes k}.
\]
This is absurd unless $k=1$, since by Theorem \ref{theorem:gtpginvariants}, $\Hbul(\gt,\C) \cong \Hbul(\g,\C)$, and $\Hbul(\g,\C)$ has dimension $2^{\text{rank}(\g)}>1$. So $\Hbul(\gI,\C) \not\cong \Hbul(\gr (\gI),\C)$ as $\gk$-modules when $k>1$.
\end{proof}

The previous lemma shows that the spectral sequence \eqref{eq:grgIspecseq} does not converge to $\Hbul(\gI,\C)$, though since the filtration on the dual Koszul complex was not exhaustive, this should not have been too surprising. The lack of a $\gk$-module isomorphism $\Hbul(\gI,\C) \cong \Hbul(\gr (\gI),\C)$ is also evident from \cref{eg:h2diff}, where we exhibited a composition factor of $\h^2(\gI,\C)$ that does not appear as a composition factor of $\h^2(\gr (\gI),\C)$.

\bibliographystyle{eprintamsmath}
\bibliography{extensions}

\end{document}